\documentclass[reqno]{amsart}

\usepackage{macros/preprintStyle}
\pdfoutput=1 %

\usepackage{bm}
\usepackage{comment}
\usepackage{adjustbox} 
\usepackage{framed}
\usepackage{xcolor}
\usepackage{amsthm}
\usepackage{tikz}
\usepackage{mathrsfs, bm}
\usepackage[scr]{rsfso}
\usepackage{dutchcal}
\usepackage{mdframed}
\usetikzlibrary{arrows.meta,positioning}
\usepackage{booktabs} 
\usepackage{multirow} 
\usepackage[most]{tcolorbox}

\newtcolorbox{assumptionbox}[1][]{
  colback=white,
  colframe=black,
  boxrule=0.4pt,   
  sharp corners,   
  fonttitle=\bfseries,
  title=#1,
  left=6pt,
  right=6pt,
  top=6pt,
  bottom=6pt
}

\usepackage{graphicx}
\usepackage{color}
\usepackage{circuitikz}
\usepackage{tikz}
\usetikzlibrary{calc,positioning,shapes}
\usetikzlibrary{patterns,decorations.pathmorphing,decorations.markings}
\usetikzlibrary{external}

\usepackage{pgfplots}
\pgfplotsset{compat=newest}
\usepackage[margin=10pt,font=small,labelfont=bf,labelsep=endash]{caption}
\usepackage{subcaption}

\usetikzlibrary{intersections, backgrounds}
\usetikzlibrary{circuits}
\usetikzlibrary{circuits.ee.IEC}
\usepackage{pgfplots}
\pgfplotsset{compat=newest}

\definecolor{color0}{rgb}{0.12156862745098,0.466666666666667,0.705882352941177}
\definecolor{color1}{rgb}{1,0.498039215686275,0.0549019607843137}
\definecolor{color2}{rgb}{0.172549019607843,0.627450980392157,0.172549019607843}
\definecolor{color3}{rgb}{0.83921568627451,0.152941176470588,0.156862745098039}
\definecolor{color4}{rgb}{0.580392156862745,0.403921568627451,0.741176470588235}
\definecolor{color5}{rgb}{0,0,0}

\definecolor{mycolor1}{rgb}{0.00000,0.44700,0.74100}
\definecolor{mycolor2}{rgb}{0.85000,0.32500,0.09800}
\definecolor{mycolor3}{rgb}{0.92900,0.69400,0.12500}
\definecolor{mycolor4}{rgb}{0.46600,0.67400,0.18800}
\definecolor{mycolor5}{rgb}{0.49400,0.18400,0.55600}

\usepackage{amssymb}

\usepackage{bm}
\usepackage{comment}
\usepackage{adjustbox} 
\usepackage{float}

\usepackage{mathtools}

\usepackage{xspace}

\usepackage{graphicx}
\usepackage{color}
\usepackage{circuitikz}
\usepackage{tikz}
\usetikzlibrary{calc,positioning,shapes}
\usetikzlibrary{patterns,decorations.pathmorphing,decorations.markings}
\usetikzlibrary{external}

\usepackage{pgfplots}
\pgfplotsset{compat=newest}
\usepackage[margin=10pt,font=small,labelfont=bf,labelsep=endash]{caption}
\usepackage{subcaption}

\usetikzlibrary{intersections, backgrounds}
\usetikzlibrary{circuits}
\usetikzlibrary{circuits.ee.IEC}
\usepackage{pgfplots}
\pgfplotsset{compat=newest}

\definecolor{color0}{rgb}{0.12156862745098,0.466666666666667,0.705882352941177}
\definecolor{color1}{rgb}{1,0.498039215686275,0.0549019607843137}
\definecolor{color2}{rgb}{0.172549019607843,0.627450980392157,0.172549019607843}
\definecolor{color3}{rgb}{0.83921568627451,0.152941176470588,0.156862745098039}
\definecolor{color4}{rgb}{0.580392156862745,0.403921568627451,0.741176470588235}
\definecolor{color5}{rgb}{0,0,0}

\definecolor{mycolor1}{rgb}{0.00000,0.44700,0.74100}
\definecolor{mycolor2}{rgb}{0.85000,0.32500,0.09800}
\definecolor{mycolor3}{rgb}{0.92900,0.69400,0.12500}
\definecolor{mycolor4}{rgb}{0.46600,0.67400,0.18800}
\definecolor{mycolor5}{rgb}{0.49400,0.18400,0.55600}

\newcommand{\smoothFunctions}[3][]{\ifthenelse{\equal{#1}{}}{\mathcal{C}^{#2}}{\mathcal{C}_{#1}^{#2}}(#3)}

\newcommand{\dist}[1][]{\ifthenelse{\equal{#1}{}}{\mathbb{D}}{#1_{\mathbb{D}}}}

\title[On the Regularity of Coupled Cluster Amplitudes]{On the Regularity and Interpolation of Coupled Cluster Amplitudes in Canonical Orbital Basis}
\author{Jonas Beck${}^\star$  \and Benjamin Stamm${}^\dagger$}

\address{${}^{\star}$ Jonas Beck, Institute of Applied Analysis and Numerical Simulation, University of Stuttgart, Pfaffenwaldring 57, 70569 Stuttgart, Germany}
\email{jonas.beck@ians.uni-stuttgart.de}
\address{${}^{\dagger}$Benjamin Stamm, Institute of Applied Analysis and Numerical Simulation, University of Stuttgart, Pfaffenwaldring 57, 70569 Stuttgart, Germany}
\email{benjamin.stamm@ians.uni-stuttgart.de}

\date{\today}

\begin{document}

\begin{abstract}
Arguably the most widely used approaches for obtaining highly accurate molecular ground-state energies are coupled cluster methods. Despite introducing two layers of approximation, a linear and a nonlinear one, coupled cluster methods remain computationally intensive, with the complexity scaling as $\mathcal{O}(\mathrm{poly}(N))$, where $N$ is the number of electrons. Moreover, this method must be applied over a large set of different nuclear coordinates in order to study certain chemical phenomena. Therefore, in this work, we investigate the regularity of single-reference coupled cluster amplitudes with respect to nuclear coordinate displacements, with the aim of enabling interpolation or extrapolation approaches that rely on only a limited number of reference geometries. We show that, in theory, under certain non-degeneracy assumptions on the Hartree–Fock level of theory, and the coupled cluster level of theory the amplitudes behave real analytic. Furthermore, we analyze the artifacts that arise in practical calculations that use canonical orbitals, which hinder this high degree of regularity, and suggest strategies to mitigate these issues. Finally, we validate our findings through numerical experiments by interpolating the amplitudes and comparing the performance of the interpolants with that of the exact amplitudes.  
\end{abstract}

\maketitle
{\footnotesize \textsc{Keywords:} Coupled cluster theory, Hartree–Fock theory, one-dimensional interpolation, electronic structure theory, quantum chemistry, parametric eigenvalue problem.} 

{\footnotesize \textsc{AMS subject classification: } Primary 81V55, 41A05; Secondary 65D05, 81-08.}


\section{Introduction}\label{sec:intro}

In modern computational chemistry, a primary objective is to solve electronic structure problems using \textit{ab initio} methods. The goal is to predict molecular behavior from first principles of quantum mechanics, without relying on costly experiments. Since the early twentieth century, with the birth of quantum mechanics, extensive efforts have been devoted to making the exponentially scaling electronic Hamiltonian accessible to such computational methods. Historically, the Hartree–Fock method was the first approach to be analyzed in depth, in the Born-Oppenheimer setting \cite{Slater, Hartree_1928, FockNherungsmethodeZL}. This method treats the complex electronic correlation as a mean field interaction, resulting in either a minimization problem of a non-linear functional or in a non-linear eigenvector problem, which ultimately provide a wave function that optimally describes the system neglecting most of the electronic interaction effects \cite{hartree1957calculation, roothaan,szabo,helgakar}.   Arguably the most successful and widely used ansatz dates back to Hohenberg and Kohn \cite{HohenbergUKohn}. Their initial works paved the way for Density Functional Theory (DFT), the state-of-the-art, electronic density–based, approach that ultimately earned Kohn and John Pople a shared Nobel Prize in 1998. Even though DFT is the most computationally efficient option to date, and often the only way to model larger systems, a key disadvantage is the lack of a constructive method to improve the modeling error at hand. Therefore, whenever highly accurate results are required, one must use wave function–based methods.  \\ 
\newline
The coupled cluster method is among the most commonly used wave function–based approaches for calculating ground-state properties of molecules. As a post-Hartree–Fock method, it systematically improves a previously obtained Hartree–Fock determinant, approaching the exact limit within the discretized space. In particular, in the single-reference regime, i.e. when static correlation is low, the coupled cluster approach is highly effective. Here, the coupled cluster ansatz takes the form of a single exponential of a cluster operator, which excites the Hartree–Fock determinant into a state that accounts for the electronic correlation effects.  Initially, the exponential ansatz for a wavefunction was developed by Hugenholtz \cite{HUGENHOLTZ1957481} and Hubbard \cite{Hubbard} in the study of many-body quantum systems. Already at this early stage of development, the \textit{size-extensivity} of such an ansatz had been demonstrated. It was later introduced into nuclear physics by Coester and Kümmel \cite{Coest,COESTER1958421}, who also provided a more accessible derivation using the terminating Baker–Campbell–Hausdorff series. Finally, the method was adapted to quantum chemistry through the work of Čížek \cite{Cizek} and others, including Paldus et al. \cite{Paldus} and Sinanoğlu \cite{Sinanoglu}. The high accuracy comes at the cost of substantial computational overhead. While the Hartree–Fock method has a time scaling of $\mathcal{O}(N^4)$, where $N$ is the system size, the full coupled cluster method exceeds any polynomial time scaling. To render the coupled cluster approach computationally feasible, an additional simplification is introduced. The cluster operators, that build the exponential ansatz, are restricted to a certain excitation level, chosen according to the user’s requirements and resources. This systematic truncation of the allowed excitations within the exponential framework of coupled cluster preserves \textit{size-consistency}, which is arguably one of the most important properties a numerical method in quantum chemistry should possess and represents a key advantage over other wave function–based approaches. By far the most commonly used are CCSD (coupled cluster singles and doubles method) and CCSD(T) (coupled cluster singles, doubles and perturbed triples method), which currently offer the best compromise between computational complexity and accuracy. The latter is widely considered the 'gold standard' and extends the CCSD ansatz by incorporating connected triple excitations through perturbative methods \cite{CCSD_perT}.\\
\newline
A rigorous mathematical analysis of the coupled cluster method started with Reinhold Schneider's work \cite{Schneider2009} on the discretized coupled cluster equations in the year 2009. Subsequently, he and his former PhD student Thorsten Rohwedder published two articles\cite{rohwedder1,rohwedder2}, that placed the continuous coupled cluster equations on a rigorous mathematical foundation and gave an investigation on the well-posedness and an error analysis based on a local, strong monotonicity assumption. An alternative approach to the well-posedness question and the error analysis via the Jacobian of the coupled cluster function was provided in the works of Hassan, Maday and Wang \cite{hassan2023analysissinglereferencecoupled,hassan2025analysissinglereferencecoupled}. Beyond the numerical analysis developed over the past 15 years, the community of mathematicians working on the numerical aspects of the electronic structure problem has adopted a range of perspectives from different mathematical disciplines to better understand the complex nonlinear nature of the coupled cluster equations, spanning areas from algebraic geometry, topological degree theory up to graph theory \cite{faulstich2024algebraicvarietiesquantumchemistry,Faulstich,Csirik_2023,Csirik_part2}.\\
\newline
In this contribution we want to deal with the coupled cluster problem in the parameterized setting. To be more precise, we aim to investigate what level of regularity can be expected for the coupled cluster amplitudes when expressed in the canonical orbital basis as functions of the nuclear positions. The amplitudes in this basis are those typically produced by virtually all modern quantum chemistry software packages when performing a coupled cluster calculation. Although the results on the Hartree–Fock level of theory can in principle be decoupled from the results on the coupled cluster level of theory, they are strongly linked in practice. Any regularity assumption needed for the latter will be directly provided by the underlying Hartree–Fock calculation. The canonical orbitals are provided by a Hartree–Fock calculation, where the Fock matrix is diagonalized and its generalized eigen functions are ordered by the \textit{aufbau principle}. Therefore, we investigate the regularity on the parameters on both levels of theory. In Section \ref{sec:setting} we review the pointwise setting of the coupled cluster method, i.e. we introduce all the necessary expressions and equations on the Hartree–Fock level of theory and post-Hartree–Fock level of theory alike. Subsequently, in Section \ref{sec:parameter_cc} we establish via Theorem \ref{IFT} and Theorem \ref{analytic_C} under what circumstances one can expect high regularity of the canonical orbitals. While Theorem \ref{IFT} is a specific version of the implicit function theorem on Riemannian manifolds and therefore requires a non-degeneracy assumption on the Riemannian Hessian, Theorem \ref{analytic_C} additionally makes use of Kato's perturbation theory in finite dimensions. Furthermore, Theorem \ref{analytic_amplitudes} shows that under a non-degeneracy assumption on the solution to the coupled cluster equations the regularity of the underlying orbitals gets inherited. The required non-degeneracy assumption is a natural extension of the standard non-degeneracy condition on eigenvalues from the linear problem, carried over to the nonlinear coupled cluster equations at any level of truncation. Motivated by the theoretical results, Section \ref{sec:interpol} introduces interpolation schemes, that put the excitation amplitudes and its regularity to the test. Although in theory one can expect very regular amplitudes in the nuclear coordinates, orbital energy crossings introduce artificial discontinuities that are addressed in Theorem \ref{MO_AO_makes_high_reg}. An algorithmic strategy is introduced that enables interpolation of the coupled cluster amplitudes. As a proof of concept, Section  \ref{NumExp} presents numerical experiments on selected amino acids based on the tensor transformation framework, which exhibit exponential error decay with an increasing number of Chebyshev nodes (see Figure \ref{fig:error_decay})). Finally, Figure \ref{fig:iter_n_decay} demonstrates that the interpolated amplitudes can be effectively used as initial guesses for the quasi-Newton method, which is commonly employed to compute coupled-cluster solutions in practice.
\section{Setting}\label{sec:setting}
In quantum chemistry the theoretical investigation of the ground state energy of a molecular system is usually done in the Born-Oppenheimer picture, where the nuclei of the molecule are treated as constant parameters and only the electron-electron and electron-nuclei interaction is taken into account. The Hamiltonian describing this setup for a molecule, with the positions of the nuclei $\lbrace R_j\rbrace_{j=1}^M $ and corresponding atomic numbers $\lbrace Z_{j}\rbrace_{j=1}^M$, is given by 

\begin{align}\label{Hamil}
    \mathcal{H}:=-\frac{1}{2}\sum^{N}_{i=1}\Delta_{x_i}+\sum^N_{i=1}V(x_i)+\sum_{i<j}\frac{1}{|x_i-x_j|}.
\end{align}
Here the Coulomb-type potential $V$ induced by the nuclei of the system is described by
\begin{align}\label{nuclei_elec_pot}
    V(x):= -\sum^{M}_{j=1}\frac{Z_j}{|x-R_j|}.
\end{align}
The electronic Hamiltonian $\mathcal{H}:\mathrm{H}^2(\mathbb{R}^{3N})\subset \mathrm{L}^2(\mathbb{R}^{3N}) \longrightarrow \mathrm{L}^2(\mathbb{R}^{3N}) $ is a bounded from below, self-adjoint operator with form domain $\mathrm{H}^1(\mathbb{R}^{3N})\subset \mathrm{L}^2(\mathbb{R}^{3N})$ (see, e.g., \cite{Kato1951FundamentalPO}). $\textbf{Note on the notation:}$ We use the notation $\langle \cdot,\cdot \rangle_{\mathrm{L}^2}$ to denote the inner product in both the one-particle space $\mathrm{L}^2(\mathbb{R}^3)$ and the many-particle space $\mathrm{L}^2(\mathbb{R}^{3N})$, with the precise meaning determined by the context. Although the qualitative structure of the spectrum of $\mathcal{H}$ is well understood in the case 
$\sum_{j=1}^M Z_j\ge N$, namely, that it separates into a discrete spectrum $\sigma_{\text{disc}}(\mathcal{H})$ consisting of eigenvalues accumulating at the ionization threshold $\Sigma$, beyond which there is only the essential spectrum $\sigma_{\text{ess}}(\mathcal{H})=[\Sigma, \infty)$ \cite{Hunz}, computing the discrete spectrum remains a major challenge. In particular the \textit{ground state energy} $\mathcal{E}_0:= \inf\sigma_{\text{disc}}(\mathcal{H})$ is of interest. From the mathematical point of view, calculating the ground state energy boils down to solving the time-independent Schrödinger equation
\begin{align}\label{tischrödi}
    \mathcal{H}\psi=\mathcal{E}_0\psi,\quad \|\psi\|_{\mathrm{L}^2}=1,
\end{align}
i.e. solving for an eigenstate $\psi \in \mathrm{H}^2(\mathbb{R}^{3N})$ and the ground state energy $\mathcal{E}_0\in \mathbb{R}$. In this work, we focus on methods that approximately solve the ground-state problem using wave function-based approaches. In contrast to the famous density functional theory (DFT), that gives only an electronic density describing the system, wavefunction methods aim to also obtain a proper eigenfunction. Since the electronic Hamiltonian \eqref{Hamil} is a very complex object, the community of quantum chemists split wave function methods in several stages of approximation, called \textit{levels of theory} in order to systematically improve the accuracy of the model by increasing the complexity. Usually the lowest level of theory is the Hartree–Fock level of theory, which scales like $\mathcal{O}(N^4)$ with system size. The Hartree–Fock method is a mean-field method and hence neglects the electron correlation. As a follow up on a Hartree–Fock calculation, either perturbative methods like \textit{Møller–Plesset perturbation theory} (MP) or methods like \textit{configuration interaction} (CI) or \textit{coupled cluster} (CC) can be applied. All these methods have in common, that they use a reference wavefunction provided by the Hartree–Fock level of theory.

\subsection{Discretization on the Hartree–Fock Level of Theory}\label{disc_HF_Theory}
The Hartree–Fock model aims to find the best approximation to the true solution of the time-independent $N$-body Schrödinger equation \eqref{tischrödi}, by means of a non-interacting system. A suitable non-interacting  system can be obtained by minimizing the \textit{Hartree-Fock energy functional}
\begin{align}\label{cont_hf_func}
    E^{\mathrm{HF}}(\varphi_1,\ldots,\varphi_N):=\sum^N_{i=1}\int_{\mathbb{R}^3}\frac{1}{2}|\nabla \varphi_i(x)|^2+V|\varphi_i(x)|^2\mathrm{d}x +\frac{1}{2}\iint_{\mathbb{R}^6}\frac{\rho(x)\rho(y)-|\rho(x,y)|^2}{|x-y|}\mathrm{d}x\mathrm{d}y,
\end{align}
where for all $i=1,\ldots,N$ we have $\varphi_i\in \mathrm{H}^1(\mathbb{R}^3)$, $\rho(x):=\sum^{N}_{i=1}|\varphi_i(x)|^2$ and $\rho(x,y):=\sum^N_{i=1}\varphi_i(x)\varphi^*_i(y)$ under the orthonormality constraint that $\int \varphi_i\varphi_j\,\mathrm{d}x =\delta_{ij}$ for $i,j=1,\ldots, N$.
After minimizing the Hartree–Fock functional \eqref{cont_hf_func}, one obtains an optimal non-interacting admissible approximation of the solution of the true $N$-body Schrödinger equation by choosing the Slater determinant \begin{align}\label{HF_slater}
    \Psi=\frac{1}{\sqrt{N!}}\mathrm{det}\big((\varphi_i(x_j))_{i,j=1}^N\big)\in \mathrm{H}^1(\mathbb{R}^{3N}).
\end{align}
The previously introduced picture on the restricted energy minimization can be considered the \textit{orbital perspective}. There is a second formulation that is widely used, which we refer to as the \textit{density perspective}. Instead of searching for a set of $\mathrm{L}^2$-orthonormal orbitals $\varphi_i \in \mathrm{H}^1(\mathbb{R}^3)$, one can rather focus on finding the orthogonal projection on the space spanned by the solution orbitals. For a fixed set $\lbrace \varphi_i\rbrace^N_{i=1}\subset \mathrm{H}^1(\mathbb{R}^3)$, the orthogonal projector onto the space spanned by these functions has the form
\begin{align*}
    \mathscr{D}=\sum_{i=1}^N \langle\varphi_i,\,\cdot\, \rangle_{\mathrm{L}^2}\,\varphi_i.
\end{align*}
For the same set of functions, we can define the associated quantities that were previously used for the Hartree–Fock functional in the orbital perspective
\begin{align*}
    \rho_{\mathscr{D}}(x,y):= \sum^{N}_{i=1}\varphi_i(x)\varphi^*_i(y),\quad \rho_{\mathscr{D}}(x)=\rho_{\mathscr{D}}(x,x)=\sum^N_{i=1}|\varphi_i(x)|^2.
\end{align*}
The Hartree–Fock energy functional in the electronic density description is
\begin{align}\label{cont_dens_HF_func}
    \mathcal{E}^{\mathrm{HF}}(\mathscr{D}):= \text{Tr}((h+\tfrac{1}{2}G(\mathscr{D}))\mathscr{D}),
\end{align}
where $h:=-\frac{1}{2}\Delta+V$ is the core Hamiltonian and 
\begin{align*}
    (G(\mathscr{D})\psi)(x):=\Big(\rho_{\mathscr{D}} \ast \frac{1}{|y|}\Big)(x)\psi(x)-\int_{\mathbb{R}^3}\frac{\psi(y)\rho_{\mathscr{D}}(x,y)}{|x-y|}\mathrm{d}y,
\end{align*}
is the electron-electron interaction operator. For the functional \eqref{cont_dens_HF_func} the minimization is carried out over the rank $N$ projectors $\mathcal{P}=\big\lbrace \mathscr{D}\in \mathscr{S}^1: \text{Ran}(\mathscr{D})\subset \mathrm{H}^1(\mathbb{R}^3),\quad \mathscr{D}^2=\mathscr{D}=\mathscr{D}^*, \quad \text{Tr}(\mathscr{D})=N\big\rbrace$, where $\mathscr{S}^1$ are the trace-class operators.\\
\newline In practice, the minimization of the energy functional \eqref{cont_hf_func} or \eqref{cont_dens_HF_func} is restricted to a finite-dimensional subspace $\mathcal{V}_{\mathrm{n}_{\text{b}}}\subset \mathrm{H}^1(\mathbb{R}^3)$, where $\text{dim}(\mathcal{V}_{\text{n}_\text{b}})=\text{n}_\text{b}<\infty$. A set of basis functions $\mathfrak{A}=\lbrace \chi_i\rbrace_{i=1}^{\text{n}_\text{b}}\subset \mathrm{H}^1(\mathbb{R}^3)$ with $\text{span}(\mathfrak{A})=\mathcal{V}_{\text{n}_\text{b}}$ is chosen as the basis of the discretization and referred to as \textit{atomic orbitals}. In general the atomic orbitals are not an $\mathrm{L}^2$-orthonormal set, hence they have a non trivial overlap, which is collected in the positive definite, symmetric  \textit{overlap matrix} $\mathbf{S}=(\mathbf{S}_{ij})_{i,j=1}^{\text{n}_\text{b}}$ with components
\begin{align}\label{overlap}
    \mathbf{S}_{ij}=\int_{\mathbb{R}^3}\chi_i^*(x)\chi_j(x)\mathrm{d}x.
\end{align}
The functional \eqref{cont_dens_HF_func} is then projected onto the atomic orbital basis and has the discretized form
\begin{align}\label{disc_en_func}
    \mathcal{E}(\mathscr{D}):=\mathcal{E}^{\mathrm{HF}}(\mathscr{D})=\text{Tr}((\mathbf{h}+\tfrac{1}{2}\mathbf{G}(\mathscr{D}))\mathscr{D}),
\end{align}
where $\mathbf{S}^{1/2}\mathbf{h}\mathbf{S}^{1/2}=\tilde{\mathbf{h}}=(\tilde{\mathbf{h}}_{ij})_{i,j=1}^{\text{n}_\text{b}}$ and $\mathbf{S}^{1/2}\mathbf{G}(\mathbf{S}^{1/2}\mathscr{D}\mathbf{S}^{1/2})\mathbf{S}^{1/2}=\tilde{\mathbf{G}}(\mathscr{D})=(\tilde{\mathbf{G}}_{ij}(\mathscr{D}))_{i,j=1}^{\text{n}_\text{b}}$ are matrices with the components
\begin{align*}
    \tilde{\mathbf{h}}_{ij}=\int_{\mathbb{R}^3}\chi_i(x)h\chi_j(x)\mathrm{d}x,\quad \tilde{\mathbf{G}}_{ij}(\mathscr{D})=\sum_{\ell,k=1}^{\text{n}_\text{b}}\mathscr{D}_{\ell k}\Big[ Q_{ij}^{k\ell }-\tfrac{1}{2}Q_{i\ell}^{kj}\Big],
\end{align*}
and with the two-body integrals
\begin{align*}
    Q_{ij}^{k\ell}= \iint_{\mathbb{R}^6} \frac{\chi^*_{i}(x)\chi_j(x)\chi^*_{k}(y)\chi_{\ell}(y)}{|x-y|}\mathrm{d}x\mathrm{d}y.
\end{align*}
In the discretized setting, we aim to minimize \eqref{disc_en_func} over the set of all \textit{density matrices}
\begin{align}\label{grass}
    \mathcal{G}:=\mathcal{G}(\text{n}_{\text{b}},N)=\big\lbrace \mathscr{D}\in \mathbb{R}^{\text{n}_\text{b}\times \text{n}_\text{b}}: \mathscr{D}^2=\mathscr{D}=\mathscr{D}^*,\quad \text{Tr}(\mathscr{D})=N\big\rbrace.
\end{align}
The set of density matrices $\mathcal{G}$ is a topological manifold, a \textit{Grassmannian manifold}, that admits a real analytic structure. Henceforth, we will always assume that the Grassmannian manifold carries a real analytic structure. Furthermore, we endow the Grassmann manifold $\mathcal{G}$ with its standard Riemannian metric $g_{\mathscr{D}}(A,B)=\text{Tr}(A^\top B) $, which is independent of $\mathscr{D}\in \mathcal{G}$, hence we omit the base point reference. Note that the overlap $\mathbf{S}$ usually is treated decoupled from the energy functional \eqref{disc_en_func}, i.e. one rather works with the matrices $\tilde{\mathbf{h}}$ and $\tilde{\mathbf{G}}$ in the functional, which than results in a manifold that depends on $\mathbf{S}$. Nevertheless, we adopt the perspective that the functional carries the explicit dependence on the overlap, such that the metric remains overlap independent. Therefore, independent of $\mathbf{S}$, the admissible densities are of the Grassmannian \eqref{grass} endowed with the standard trace metric.  In the chemistry community the energy functional with respect to $\mathbf{h}$ and $\mathbf{G}$ is considered to be in \textit{orthogonal atomic orbital basis}. Furthermore, the Grassmannian is a Riemannian submanifold of the flat space    $\mathbb{R}^{\text{n}_\text{b}\times \text{n}_\text{b}}$ equipped with the trace metric. For any $\mathscr{D}\in\mathcal{G}$ the corresponding tangent space is given by
\begin{align}
    T_\mathscr{D}\mathcal{G}=\big\lbrace X \in \mathbb{R}^{\text{n}_\text{b}\times \text{n}_\text{b}}: X^\top=X, \quad \mathscr{D}X+X\mathscr{D}=X\big\rbrace.
\end{align}
The fact that $\mathcal{G}\subset \mathbb{R}^{\text{n}_\text{b}\times \text{n}_\text{b}} $ is a Riemannian submanifold, provides at any base point $\mathscr{D}\in \mathcal{G}$ the orthogonal decomposition 
\begin{align*}
    T_\mathscr{D}\mathcal{G}\oplus \big(T_\mathscr{D}\mathcal{G}\big)^\perp = T_\mathscr{D}\mathbb{R}^{\text{n}_\text{b}\times \text{n}_\text{b}} \cong \mathbb{R}^{\text{n}_\text{b}\times \text{n}_\text{b}},
\end{align*}
where $\big(T_\mathscr{D}\mathcal{G}\big)^\perp=\big\lbrace Y\in \mathbb{R}^{\text{n}_\text{b}\times \text{n}_\text{b}}: \text{Tr}(Y^\top X)=0 \,\,\,\text{for all}\,\,\, X\in T_\mathscr{D}\mathcal{G}\, \big \rbrace$. This decomposition allows at any point $\mathscr{D}\in \mathcal{G}$ to project from the ambient space $\mathbb{R}^{\text{n}_\text{b}\times \text{n}_\text{b}}$ onto the tangent space $T_\mathscr{D}\mathcal{G}$ via the projection operator
\begin{align}\label{proj}
   P_\mathscr{D}: \mathbb{R}^{\text{n}_\text{b}\times \text{n}_\text{b}} \longrightarrow T_\mathscr{D}\mathcal{G}, \quad X \mapsto [\mathscr{D},[\mathscr{D},X]].
\end{align}
At this point we can come back to the discretized Hartree–Fock functional in atomic orbital description. Note that we have not yet completely specified the domain of the functional \eqref{disc_en_func}. We define the ambient energy functional
\begin{align}\label{ambient energy functional}
    E:\mathbb{R}^{\text{n}_\text{b}\times \text{n}_\text{b}}\longrightarrow \mathbb{R}, \quad X \mapsto \text{Tr}((\mathbf{h}+\tfrac{1}{2}\mathbf{G}(X))X), 
\end{align}
and the energy functional on the Grassmannian $\mathcal{E}:\mathcal{G}\longrightarrow \mathbb{R}$ as the restriction $\mathcal{E}=E\,|_\mathcal{G}$. A necessary condition for $\mathscr{D}\in \mathcal{G}$ to be a minimizer of $\mathcal{E}$ one requires
\begin{align*}
    \nabla_\mathcal{G}\mathcal{E}(\mathscr{D})=0\in T_\mathscr{D}\mathcal{G},
\end{align*}
where $\nabla_\mathcal{G}\mathcal{E}:\mathcal{G}\longrightarrow T\mathcal{G}$ is the Riemannian gradient. The submanifold structure of $\mathcal{G}$ gives at any $\mathscr{D}\in \mathcal{G}$ a simple connection between both gradients via \eqref{proj}, i.e.
\begin{align}
    \nabla_\mathcal{G}\mathcal{E}(\mathscr{D})=P_\mathscr{D}\nabla E(\mathscr{D}).
\end{align}
Since the flat gradient $\nabla E(\mathscr{D})=\mathbf{F}(\mathscr{D})=\mathbf{h}+\mathbf{G}(\mathscr{D})$ is the well known Fock matrix, we have for the Riemannian gradient
\begin{align}\label{riem_grad}
    \nabla_\mathcal{G}\mathcal{E}(\mathscr{D})=P_\mathscr{D} \mathbf{F}(\mathscr{D})=[\mathscr{D},[\mathscr{D},\mathbf{F}(\mathscr{D})]].
\end{align}

\begin{lemma}\label{Root_dens_corr}
    A density matrix $\mathscr{D}\in \mathcal{G}$ is a critical point to the energy functional $\mathcal{E}:\mathcal{G}\longrightarrow \mathbb{R}$, defined through \eqref{disc_en_func}, if and only if there exists $C, \Lambda\in \mathbb{R}^{\mathrm{n}_\mathrm{b}\times \mathrm{n}_\mathrm{b}}$, such that $\mathscr{D}$ is the projector onto the first $N$ columns of $C$ and
    \begin{align}\label{roothaans_orth}
        \mathbf{F}(\mathscr{D})C=C\Lambda, \quad C^\top C=\mathrm{I}.
    \end{align}
\end{lemma} A proof for a result analogous to Lemma \ref{Root_dens_corr} can be found in e.g. \cite{refId0}. Note that \eqref{roothaans_orth} corresponds to the \textit{Roothaan equations} in the orthogonal atomic orbital basis. The solution matrix $C$ is referred to as the \textit{orthogonal coefficient matrix}. An advantage of the Hartree–Fock method over density functional theory, is that one not only obtains a density $\mathscr{D}\in \mathcal{G}$ but also a wavefunction that approximately describes the ground state of the fermionic system. The wavefunction is obtained by selecting a corresponding representative of the density matrix from the Stiefel manifold. In practice this can be done by solving the Euler-Lagrange equations \eqref{roothaans_orth} corresponding to the minimization problem of the Hartree–Fock functional. Typically, one orders the eigenvectors in increasing order by its eigenvalues, counting multiplicities, i.e. the entries of the diagonal matrix $\Lambda =\text{diag}(\lambda_1,\ldots,\lambda_{\text{n}_\text{b}}) $, containing the orbital energies, is ordered according to
\begin{align}
    \lambda_1\le \lambda_2\le\ldots\le \lambda_{\text{n}_\text{b}}.
\end{align}
This choice of gauge is referred to as the \textit{aufbau principle}. The coefficient matrix $C$ containing the orthonormalized eigenvectors is then split in two parts, the matrix $C_{\text{occ}}\in \mathbb{R}^{\text{n}_\text{b}\times N}$ containing the \textit{occupied} orbital coefficients and $C_{\text{virt}}\in \mathbb{R}^{\text{n}_\text{b}\times(\text{n}_\text{b}-N) }$ containing the \textit{virtual} orbital coeffcients

\begin{align}\label{occ_virt_C_matrix}
C = \left(
\begin{array}{cccc|ccc}
C_{11} & C_{12} & \cdots & C_{1\,N} & C_{1\,N+1} & \cdots & C_{1\text{n}_{\text{b}}} \\
C_{21} & C_{22} & \cdots & C_{2\,N} & C_{2\,N+1} & \cdots & C_{2\text{n}_{\text{b}}} \\
\vdots & \vdots & \ddots & \vdots & \vdots & \ddots & \vdots \\
C_{\text{n}_{\text{b}} 1} & C_{\text{n}_{\text{b}} 2} & \cdots & C_{\text{n}_{\text{b}}\,N} & C_{\text{n}_{\text{b}}\,N+1} & \cdots & C_{\text{n}_{\text{b}} \text{n}_{\text{b}}}
\end{array}
\right)
\nonumber \\
\underbrace{\hspace{4cm}}_{\displaystyle C_\mathrm{occ}}
\underbrace{\hspace{3.8cm}}_{\displaystyle C_\mathrm{virt}}\,\,
\end{align}
Here, it is important to note that it is necessary to have a HOMO-LUMO gap $\gamma>0$, i.e the $N$th and $(N+1)$th orbital energies satisfy
\begin{align}
    \lambda_{N+1}-\lambda_N\ge\gamma>0.
\end{align}
Otherwise, it is impossible to distinguish between the occupied columns and the virtual columns of \eqref{occ_virt_C_matrix}. 
The coefficient matrix provides the transformation from the atomic orbitals $\mathfrak{A}$, which we chose in the very beginning, to the so called \textit{molecular orbitals} $\mathfrak{M}$. While the atomic orbitals can be considered a matter of choice, the molecular orbitals, or at least the occupied molecular orbitals, provide us the Hartree–Fock Slater determinant \eqref{HF_slater} in the framework of our discretization choice. First and foremost we build the molecular orbitals $\mathfrak{M}$ via the \textit{linear combination of atomic orbitals} ansatz (LCAO), i.e. $\mathfrak{M}=\lbrace\psi_i\rbrace_{i=1}^{\text{n}_\text{b}}\subset \text{span}\,\mathfrak{A}$ with

\begin{align}
    \psi_i=\sum_{j=1}^{\text{n}_\text{b}}C_{ji}\chi_j,\qquad \qquad(\forall i\le \text{n}_\text{b}),
\end{align}
where the coefficients are exactly the coefficients provided by the matrix $C$ \eqref{occ_virt_C_matrix} obtained as a solution to the Hartree–Fock method in the orthogonal orbital setting. We further distinguish between occupied and virtual molecular orbitals in the same manner as we did with the $C$ matrix itself. A molecular orbital $\psi_i$ is called an \textit{occupied orbital} if the coefficient column belongs to $C_{\text{occ}}$ and virtual if the opposite is the case. Therefore, we write $\mathfrak{M}=\lbrace \psi_i\rbrace_{i=1}^{N}\cup\lbrace \psi_a\rbrace_{a=N+1}^{\text{n}_\text{b}}$, where $\psi_q$ corresponds to the $q$th-column of $C$ for any $q\le \text{n}_\text{b}$. Note that as an usual convention in the quantum chemist community one denotes by 
\begin{align}\label{indixeing_occ_virt}
    a,b,\ldots \text{virtual indices, by\,\,} i,j,\ldots \text{occupied indices and by }p,q,\ldots \text{indices of any sort.}
\end{align}
For any choice of $N$ molecular orbital functions $\lbrace \psi_{q_1},\ldots\psi_{q_N} \rbrace$ contained in $\mathfrak{M}$, we can build the Slater determinant \begin{align}\label{slater_det}
    \Psi_\nu= \frac{1}{\sqrt{N!}}\sum_{\tau\in S(N)}\text{sgn}(\tau) \prod_{\ell=1}^N \psi_{q_\ell}\big(r_{\tau(\ell)}\big),
\end{align}
where $S(N)$ denotes the symmetric group of permutations of $N$ elements and $\nu=(I,A)$ is a so called \textit{excitation index}. Here $I=(i_1,\ldots,i_r)$ with $i_1<\ldots<i_r\le N$ are the indices of the occupied orbitals that are not contained in $\lbrace \psi_{q_1},\ldots,\psi_{q_N}\rbrace$   and $A=(a_1,\ldots,a_r)$ with $N<a_1<\ldots<a_r\le \text{n}_\text{b}$ are the indices in the virtual orbitals that were used to built the determinant. The number $0\le r \le N$ is the excitation rank or the excitation index of the determinant. The set of all excitation indices of rank $R$ is denoted by $\mathscr{I}^{(r)}$ and the total set of all  (true) excitation indices by $\mathscr{I}:=\cup_{0<r\le N}\mathscr{I}^{(r)}$. We exclude the special case of excitation rank $R=0$, since the resulting determinant corresponds to the Hartree–Fock determinant that is the solution to the mean-field approximation and is in the following called the \textit{reference determinant} $\Psi_0$. The set of all Slater determinants is denoted by
\begin{align}\label{Slater_determinants}
    \mathcal{S}(\mathfrak{M})=\lbrace \Psi_0\rbrace \cup \lbrace \Psi_\nu\,|\, \nu\in \mathscr{I}\rbrace.
\end{align} 
While $\mathfrak{A}$ and also $\mathfrak{M}$ were both basis for discretizing the one particle space $\mathrm{H}^1(\mathbb{R}^3)$, the $\text{span}_\mathbb{C}\, \mathcal{S}(\mathfrak{M})$ is the finite dimensional subspace that discretizes the $N$-body space $\mathrm{H}^1(\mathbb{R}^{3N})$. Note that by construction the Galerkin discretization via Slater determinants obeys the Pauli exclusion principle, i.e. only anti-symmetric functions are considered.

\subsection{Discretization on the post-Hartree–Fock Level of Theory}

The Hartree–Fock method provides only a mean-field approximation of the system and therefore loses accuracy as electron correlation increases. Hence, for most fermionic systems of practical interest, more sophisticated methods are required. Post-Hartree–Fock methods use the Hartree–Fock state as a reference and systematically incorporate correlation effects. Within the finite basis set already chosen at the Hartree–Fock level of theory (where molecular orbitals are built from atomic orbitals), these systematic corrections ultimately converge to the exact discrete solution, which is usually referred to as the \textit{Full Configuration Interaction limit}. In order to introduce these methods, we adopt the formalism of second quantization. 
\subsubsection{Second Quantization}
Let $\mathfrak{h}$ be the separable, complex Hilbert space that describes our one particle fermionic system. The second quantization framework focuses on the fermionic Fock space generated by $\mathfrak{h}$, i.e.
\begin{align}
    \mathfrak{F}(\mathfrak{h})=\bigoplus_{n=0}^{K} \mathfrak{h}^{\wedge n},
\end{align}
where $\wedge$ denotes the antisymmetric tensor, $K< \infty$ and $\mathfrak{h}^{\wedge 0}:= \mathbb{C}$ is the vacuum. For our purposes the choice of $\mathfrak{h}$ is provided by our choices on the underlying Hartree–Fock level of theory. As explained in Subsection \ref{disc_HF_Theory}, the discretization of the one-particle fermionic system, for Hartree–Fock, starts by the choice of atomic orbitals $\mathfrak{A}$. Therefore, the Hilbert space of choice will be $\mathfrak{h}:=\text{span}_{\mathbb{C}}\,\mathfrak{A}$. One of the key advantages of second quantization is the introduction of \textit{annihilation} and \textit{creation operators}. For any element $f \in \mathfrak{h}$ there exists a corresponding pair of bounded, adjoint operators, the creation operator $a^{\dagger}(f):\mathfrak{F}(\mathfrak{h})\longrightarrow \mathfrak{F}(\mathfrak{h})$ and annihilation operator $a(f):\mathfrak{F}(\mathfrak{h})\longrightarrow \mathfrak{F}(\mathfrak{h})$. Since the Hilbert space $\mathfrak{h}\subset \mathrm{H}^1(\mathbb{R}^{3})$ is in our case a finite dimensional function space the action of a creation operator can be explicitly described via an outer product. For a function $\Psi\in \mathfrak{h}^{\wedge L}$ and a function $f\in \mathfrak{h}$ with corresponding creation operator $a(f)$, the action of creation has the form of
\begin{align}
    \big(a^{\dagger}(f)\Psi\big)(r_1,\ldots,r_{L+1})=\sum_{\tau \in S(L+1)}\frac{\text{sgn}(\tau)}{\sqrt{L!(L+1)!}}f\big(r_{\tau(1)}\big)\Psi\big(r_{\tau(2)},\ldots, r_{\tau(L+1)}\big).
\end{align}
 The adjoint action of the annihilation operator $a(f)$ corresponding to $f\in\mathfrak{h}$ on $\Psi\in\mathfrak{h}^{\wedge L}$ is given by \begin{align}
    \big(a(f)\Psi\big)(r_1,\ldots r_{L-1})=\sqrt{L}\int_{\mathbb{R}^3} f(r)\Psi(r,r_1,\ldots,r_{L-1})\mathrm{d}r.
\end{align} Interpreting the summands of the Fock space in the usual sense, i.e. $\mathfrak{h}^{\wedge n}$ describes the $n$-particle system of fermions, then the creation operator creates another fermion, and the annihilation operator will remove a fermion. Since the operator $a^{\dagger}:\mathfrak{h}\longrightarrow \mathcal{B}(\mathfrak{F}(\mathfrak{h}))$ is a bounded operator, it suffices to pick a basis of $\mathfrak{h}$ to describe all meaningful creations and annihilations in the Fock space $\mathfrak{F}(\mathfrak{h})$. Here the set of molecular orbitals $\mathfrak{M}$, which is determined on the Hartree–Fock level of theory, comes into play. We use the previously introduced splitting  $\mathfrak{M}=\lbrace{\psi}_i\rbrace_{i=1}^{N}\cup \lbrace \psi_a\rbrace_{a=N+1}^{\text{n}_{\text{b}}}$ and build all the creation and annihilation operators 
\begin{align}\label{calc_annihi,crea}
    a^{\dagger}(\psi_i):=a^\dagger_i,\qquad a(\psi_i):=a_i,
\end{align}
such that we end up with the sets $\lbrace a^\dagger_i\rbrace_{i=1}^{N}\cup \lbrace a^\dagger_a\rbrace_{a=N+1}^{\text{n}_{\text{b}}}$ and $\lbrace a_i\rbrace_{i=1}^{N}\cup \lbrace a_a\rbrace_{a=N+1}^{\text{n}_{\text{b}}}$. It is now possible to write the discretization $\mathcal{H}_\mathfrak{M}$ of the full electronic Hamiltonian \eqref{Hamil} with respect to $\mathfrak{M}$ with the corresponding creation and annihilation operators as 
\begin{align}\label{sec_quan_Hamil}
    \mathcal{H}_\mathfrak{M}=\sum_{p,q=1}^{\text{n}_\text{b}} h_{pq}a^\dagger_pa_q+\frac{1}{2}\sum^{\text{n}_\text{b}}_{p,q,r,s=1}q^{rs}_{pq}a^\dagger_p a^\dagger_ra_sa_q,
\end{align}
where $h_{pq}$ and $q^{rs}_{pq}$ are the interaction integrals appearing in the energy functional \eqref{disc_en_func}, but in the orthogonal molecular orbital basis. Using the orthonormal molecular orbitals $\mathfrak{M}$ instead of the atomic orbitals $\mathfrak{A}$, that usually have a non-trivial overlap, leads to the \textit{canonical anti-commutation relations} (CAR), i.e.
\begin{align}
    [a_i^\dagger,a_j]_+=\delta_{ij},\qquad [a^\dagger_i,a^\dagger_j]_+=[a_i,a_j]_+=0,
\end{align}
where $[A,B]_+=AB+BA$ is the typical anti-commutator of two operators $A$ and $B$.

\subsubsection{Excitation and Cluster Operators} When modeling a molecular quantum system using the machinery of second quantization, particularly with creation and annihilation operators, it is important to remember that, in most cases, the total number of electrons $N$ in the system is conserved. Therefore, for any creation of a new electron there has to be a suitable annihilation of another particle and vice versa. This is achieved by introducing \textit{excitation operators}. Let us again stick to the framework were we have already determined a fixed set of molecular orbitals $\mathfrak{M}$, that is divided in occupied and virtual orbitals. Moreover, we consider the corresponding sets of creation and annihilation operators given by \eqref{calc_annihi,crea}. An excitation operator is the product of creation and annihilation operators 
\begin{align}
    \mathscr{X}^A_I=a^\dagger_{a_r} \cdot \ldots \cdot a^\dagger_{a_1} a_{i_r}\cdot \ldots \cdot a_{i_1},
\end{align}
where $I:=(i_1,\ldots,i_r)\in \mathbb{N}^r$ and $A:=(a_1,\ldots,a_r)\in \mathbb{N}^r$ are multi-indices with 
\begin{align}\label{index_ordering}
    0<i_1<\ldots<i_r\le N <a_1<\ldots<a_r\le\text{n}_\text{b}.
\end{align}
Two indices $I$ and $A$, which satisfy the ordering \eqref{index_ordering}, provide a unique excitation index $\nu:=(I,A)$. Henceforth, any excitation operator can be written as
\begin{align}\label{exc_op}
    \mathscr{X}_\nu:=\mathscr{X}^{a_1\ldots a_r}_{i_1\ldots i_r}:=\mathscr{X}^A_I:=a^\dagger_{a_r} \cdot \ldots \cdot a^\dagger_{a_1} a_{i_r}\cdot \ldots \cdot a_{i_1},
\end{align}
for some $r\le N$. We can now write any of the Slater determinants \eqref{slater_det} by means of an excitation of the reference determinant, i.e. 
\begin{align}
    \Psi_\nu=\mathscr{X}_\nu\Psi_0, \qquad\qquad(\forall \nu \in \mathscr{I}).
\end{align}
Hence, the Slater basis that is built in \eqref{Slater_determinants} can be also recovered via
\begin{align}\label{slater_basis_via_exc}
    \mathcal{S}(\mathfrak{M})=\lbrace \Psi_0\rbrace \cup \lbrace \mathscr{X}_\nu\Psi_0\,|\, \nu\in \mathscr{I}\rbrace.
\end{align}
The $\mathrm{L}^2$-adjoint of an excitation operator $X_\nu$ is called a \textit{de-excitation operator} and is denoted by $X_\nu^\dagger$.
Further useful properties of the set of excitation operators, that we will use later, are summarized in the following proposition. 
\begin{proposition}\label{proposition_excitation_comm}
    For any $\nu, \eta \in \mathscr{I}$ it holds that 
    \begin{enumerate}
        \item (commutativity) $[\mathscr{X}_\nu,\mathscr{X}_\eta]=\mathscr{X}_\nu \mathscr{X}_\eta- \mathscr{X}_\eta \mathscr{X}_\nu = 0=\mathscr{X}^\dagger_\nu \mathscr{X}^\dagger_\eta- \mathscr{X}^\dagger_\eta \mathscr{X}^\dagger_\nu=[\mathscr{X}^\dagger_\nu,\mathscr{X}^\dagger_\eta]$.
        \item (closedness) either $\mathscr{X}_\nu \mathscr{X}_\eta=0$ or there exists $\upsilon \in \mathscr{I}$ with $\mathscr{X}_\nu \mathscr{X}_\eta=\mathscr{X}_\upsilon$.
        \item (nilpotency)  $\big(\mathscr{X}^\dagger_\nu\big)^2=\mathscr{X}_\nu^2=0$.
    \end{enumerate}
\end{proposition}
The proof for Proposition \ref{proposition_excitation_comm} was originally established in \cite{Schneider2009}. Due to \eqref{slater_basis_via_exc} the $N$-body discretization space $\mathfrak{S}(\mathfrak{M}):=\text{span}_\mathbb{C}\,\mathcal{S}(\mathfrak{M})$ is isomorphic to the space
\begin{align}
    \mathfrak{C}(\mathfrak{M}):=\Big\lbrace \sum_{\nu\in \mathscr{I}}t_\nu \mathscr{X}_\nu \,\Big|\, t_\nu\in \mathbb{C}\,\text{ for all }\, \nu\in \mathscr{I}\Big\rbrace
\end{align}
of \textit{cluster operators}. This relation gives rise to the starting perspective of many post-Hartree–Fock methods. 
\subsubsection{Coupled Cluster Method}
Consider the weak Schrödinger equation in the discretized space, i.e. 
\begin{align}\label{disc_weak_schrödi}
    \langle \Phi, \mathcal{H}_\mathfrak{M}\Psi\rangle_{\mathrm{L}^2}=E_0 \langle \Phi,\Psi\rangle_{\mathrm{L}^2},\qquad\qquad(\forall \Phi \in \mathfrak{S}(\mathfrak{M})),
\end{align}
where $\Psi \in \mathfrak{S}(\mathfrak{M})$ and $E_0\in \mathbb{R}$ is the approximate ground state energy. Since $\mathfrak{S}(\mathfrak{M})\cong \mathfrak{C}(\mathfrak{M})$ we can take the Hartree–Fock determinant $\Psi_0$ and reformulate \eqref{disc_weak_schrödi} to
\begin{align}\label{full_ci}
    \langle \Phi, \mathcal{H}_\mathfrak{M}(\mathrm{I}+T)\Psi_0\rangle_{\mathrm{L}^2}=E_0 \langle \Phi,(\mathrm{I}+T)\Psi_0\rangle_{\mathrm{L}^2},\qquad\qquad(\forall \Phi \in \mathfrak{S}(\mathfrak{M})),
\end{align}
where besides $E_0$ now the cluster operator $T\in \mathfrak{C}(\mathfrak{M})$ is the unknown. Solving the ground-state problem using a linear parameterization with cluster operators as in \eqref{full_ci}, constructed from molecular orbitals obtained through a Hartree–Fock calculation, is known as the \textit{full configuration interaction} (FCI) method and is equivalent to solving the discrete Schrödinger equation. Despite the FCI method is computationally intractable for nearly all realistic systems, it provides a natural and systematic ansatz for reducing the computational cost. This reduction in cost is achieved by truncating the set of admissible excitations. For any $\ell\le N$ define the \textit{excitation index set of level} $\ell$
\begin{align}\label{trun_ex_ind_set}
    \mathscr{I}_\ell:=\bigcup_{r\le \ell} \mathscr{I}^{(r)}.
\end{align} 
The truncated set of cluster operators is denoted by $\mathfrak{C}_\ell(\mathfrak{M})\subset \mathfrak{C}(\mathfrak{M})$ and the truncated set of Slater determinants, is denoted by
\begin{align}\label{trunc_slater_basis}
    \mathcal{S}_\ell(\mathfrak{M}):=\lbrace \Psi_0\rbrace \cup \lbrace \mathscr{X}_\nu \Psi_0\,|\, \nu \in \mathscr{I}_\ell\rbrace.
\end{align}
Only considering cluster operators in $\mathfrak{C}_\ell(\mathfrak{M})$ reduces the degrees of freedom and expected accuracy of the respective truncated approximate ground state energy $E_{0,\ell}$ at the same time. A big downside of this reduction scheme is the loss of \textit{size consistency}, which is methodologically an important feature for quantum chemistry \cite{helgakar}. This downside can be circumvented by moving from a linear parameterization to a certain non-linear one. 
\begin{theorem}\label{CC_works}
    Let $\Psi_0$ be the reference determinant of a fixed set of molecular orbitals $\mathfrak{M}$. For any $\Psi \in \mathfrak{S}(\mathfrak{M})$ that is intermediate normalized, i.e. $\langle \Psi_0,\Psi\rangle_{\mathrm{L}^2}=1$, there exists a unique cluster operator $T\in \mathfrak{C}(\mathfrak{M})$ such that $\Psi=\mathrm{e}^T\Psi_0$.
\end{theorem}
A rigorous proof for the discretized setting can be found in \cite{Schneider2009} and a similar theorem for the continuous setting in \cite{rohwedder1}. There are two main advantages to move from the linear parameterization \eqref{full_ci} to the weak formulation of the equivalent form
\begin{align}\label{exp_ansatz}
    \langle \Phi, \mathcal{H}_\mathfrak{M}\mathrm{e}^T\Psi_0\rangle_{\mathrm{L}^2}=E_0 \langle \Phi,\mathrm{e}^T\Psi_0\rangle_{\mathrm{L}^2},\qquad\qquad(\forall \Phi \in \mathfrak{S}(\mathfrak{M})),
\end{align}
where one solves for $T\in \mathfrak{C}(\mathfrak{M})$ and $E_0\in \mathbb{R}$. Both advantages are related to the process of excitation index truncation for cheaper calculations. First and foremost, the exponential ansatz \eqref{exp_ansatz} remains size consistent even if we only consider cluster operators $T\in \mathfrak{C}_\ell(\mathfrak{M})$ for $\ell<N$ \cite{helgakar, Schneider2009}. The second advantage is that, due to its nonlinearity, the ansatz implicitly accounts for excitations higher than rank $\ell$, resulting in much greater accuracy compared to the truncated linear ansatz. The solving of the Schrödinger equation with the exponential parameterization \eqref{exp_ansatz} is referred to as the \textit{full coupled cluster} (FCC) method. In practice, two modifications are made in the FCC method, that are motivated from the algorithmic point of view. On the one hand it suffices to test the solution only against the Slater basis $\mathcal{S}(\mathfrak{M})$, on the other hand it was rigorously proved in \cite{Schneider2009} that the problem can be equivalently solved after applying $\mathrm{e}^{-T}$ to the left of the Hamiltonian. This gives rise to the \textit{linked coupled cluster equations}
\begin{align}\label{linked_FFC}
    \langle \Psi_\nu, \mathrm{e}^{-T}\mathcal{H}_\mathfrak{M}\mathrm{e}^T\Psi_0\rangle_{\mathrm{L}^2}=E_0\langle\Psi_\nu,\Psi_0 \rangle_{\mathrm{L}^2},\qquad\qquad(\forall \nu\in \mathscr{I}\cup\lbrace0\rbrace).
\end{align}
Multiplying $\mathcal{H}_\mathfrak{M}$ with $\mathrm{e}^{-T}$ from the left results in a finitely terminating Baker–Campbell–Hausdorff expansion \cite{crawford}, i.e.
\begin{align*}
    \mathrm{e}^{-T}\mathcal{H}_\mathfrak{M}\mathrm{e}^T=\mathcal{H}_\mathfrak{M}+[\mathcal{H}_\mathfrak{M},T]+\frac{1}{2}[[\mathcal{H}_\mathfrak{M},T],T]+\frac{1}{6}[[[\mathcal{H}_\mathfrak{M},T],T],T]+\frac{1}{24}[[[[\mathcal{H}_\mathfrak{M},T],T],T],T].
\end{align*}
Hence \eqref{linked_FFC} is a system of quartic polynomial equations in the cluster operator $T$. The disadvantage of this modification is the loss of symmetry in the operator. Therefore, the coupled cluster method is a non-variational method. Note that the orthonormality of the molecular orbitals $\mathfrak{M}$ implies orthonormality of the Slater basis $\mathcal{S}(\mathfrak{M})$, so the linked coupled cluster equations actually read
\begin{align}
    \langle \Psi_\nu, \mathrm{e}^{-T}\mathcal{H}_\mathfrak{M}\mathrm{e}^T\Psi_0\rangle_{\mathrm{L}^2}=E_0 \delta_{ 0,\nu},\qquad\qquad(\forall \nu\in \mathscr{I}\cup\lbrace0\rbrace),
\end{align}
where $\delta_{i,j}$ is the Kronecker delta. Since FCC, just like FCI, is computationally intractable, we systematically can truncate the excitation index set $\mathscr{I}$ in order to trade complexity for accuracy. As already mentioned, the truncation process should not result in a loss of the size consistency. This requirement is met by truncating according to the excitation rank, i.e. considering an excitation index set \eqref{trun_ex_ind_set} of lower level than $N$. Truncating in arbitrary manner does in general not guarantee that the coupled cluster method stays size consistent. The \textit{truncated, linked coupled cluster equations} are
\begin{align}\label{trun_cc_eq}
    \langle \Psi_\nu, \mathrm{e}^{-T}\mathcal{H}_\mathfrak{M}\mathrm{e}^T\Psi_0\rangle_{\mathrm{L}^2}=0,\qquad\qquad(\forall \nu\in \mathscr{I}_\ell),
\end{align}
where $T\in \mathfrak{C}_\ell(\mathfrak{M})$. After a suitable $T\in \mathfrak{C}_\ell(\mathfrak{M})$ is found, which solves the truncated coupled cluster equations, we define the \textit{coupled cluster energy}
\begin{align}\label{cc_energy}
     E_{\ell}:=E_{CC}:=\langle \Psi_0, \mathrm{e}^{-T}\mathcal{H}_\mathfrak{M}\mathrm{e}^T\Psi_0\rangle_{\mathrm{L}^2}.
\end{align}
The excitation index set truncation destroyed the equivalence to the regular weak eigenvalue problem \eqref{disc_weak_schrödi}. Therefore, the energy $E_{CC}$ is rather defined as the quantity that approaches the true eigenvalue $E_0$ as $\ell \rightarrow N$. Note that the coupled cluster energy depends on the solution operator $T\in\mathfrak{C}_\ell(\mathfrak{M})$. It is well known that there can be many solutions to \eqref{trun_cc_eq} and hence possibly many different energies $E_{CC}$, see e.g. \cite{faulstich2024algebraicvarietiesquantumchemistry}. The equations \eqref{trun_cc_eq} can be formulated via the \textit{coupled cluster function} defined by
\begin{align}\label{cc_function}
    \mathcal{Q}: \mathbb{C}^{|\mathscr{I}_\ell|}\longrightarrow \mathbb{C}^{|\mathscr{I}_\ell|},\quad t\mapsto \mathcal{Q}(t),
\end{align}
where $\mathcal{Q}_\nu(t)=\langle \Psi_\nu, \mathrm{e}^{-T(t)}\mathcal{H}_\mathfrak{M}\mathrm{e}^{T(t)}\Psi_0\rangle_{\mathrm{L}^2}$, and the postfix notation $T(t)\in \mathfrak{C}_\ell (\mathfrak{M})$ emphasizes that the cluster operator is built with the coefficients provided by $t\in \mathbb{C}^{|\mathscr{I}_\ell|}$. Henceforth, we will stick to the cluster function $\mathcal{Q}$, where the coupled cluster equations \eqref{trun_cc_eq} reduce to an ordinary root finding problem
\begin{align}\label{root_eq_cc_fun}
    \mathcal{Q}(t)=0,
\end{align}
of a quartic polynomial function in $t$. Solutions to the root finding problem are called the \textit{coupled cluster amplitudes}.  In modern software \eqref{root_eq_cc_fun} is solved iteratively by a quasi-Newton method, that typically uses the MP2 amplitudes as a starting guess.

\section{The Parameter Dependent Problem}\label{sec:parameter_cc}In this work, we are interested on rigorous results concerning the expected regularity of the coupled cluster amplitudes, that is, the solutions of \eqref{root_eq_cc_fun}, with respect to nuclear displacements. Since the coupled cluster method is a post-Hartree–Fock method, this is strongly linked to the reference quantities provided by the underlying Hartree–Fock level of theory. However, the reference framework can be decoupled from the Hartree–Fock method, allowing one to simply assume that both the reference determinant and the excited determinants possess high regularity with respect to nuclear displacements. We chose not to pursue this direction for several reasons. The convergence rate of Newton-type methods for solving the root-finding problem \eqref{root_eq_cc_fun} depends strongly on the chosen reference framework (see, e.g., \cite{Schneider2009, helgakar}). In particular, the reference determinant $\Psi_0$ should have as large an overlap as possible with the true solution. Therefore, in practice, one can not simply choose the reference framework arbitrary and in modern quantum chemistry software the most reliable choice is still the Hartree–Fock reference. Hence, as a necessary step we will also prove that, under some non-degeneracy assumption on the Riemannian Hessian, the Hartree–Fock framework has locally high regularity in the nuclear coordinates.\\
\newline
Henceforth, we denote the nuclear coordinate domain by $\Omega\subset\mathbb{R}^{3M}$, where $M>0$ is the number of atoms. Due to the physical behavior of the atoms, we can assume that the position of any two nuclei never collide. For simplicity one can think of the configuration space as
\begin{align}\label{multidim_coord_sp}
    \Omega=\Omega_1\times \ldots \times \Omega_M,
\end{align}
where $\Omega_j\subset \mathbb{R}^3$ is an open subset for all $j\le M$ and $\Omega_i\cap\Omega_j=\emptyset$ for $i\ne j$. Since we will later mostly restrict our setting to an analytic curve evolving in the nuclear coordinate space, the restrictions can be formulated differently in that context, effectively allowing intersections of the regions.
\begin{definition}\label{trajectory_nuc_spac}
    Let $I\subset \mathbb{R}$ be a bounded interval. A \textit{trajectory} through the nuclear coordinate space, is a map 
    \begin{align}
        \Gamma:I\longrightarrow \mathbb{R}^{3M},\quad \mu\mapsto \Gamma(\mu)=(\Gamma_1(\mu),\ldots,\Gamma_M(\mu)),
    \end{align}
    where $\Gamma_j(\mu) \in \mathbb{R}^3$ for all $\mu\in I$ and $j\le M$, such that $\Gamma_i(\mu)\ne \Gamma_j(\mu)$ for all $\mu \in I$ and $i\ne j$.
\end{definition}
Note that the parameterization of the change of the nuclear geometry of a molecule by a trajectory $\Gamma$ arises naturally, for example, in molecular dynamics, where the free parameter $\mu$ represents time. Definition \ref{trajectory_nuc_spac} only demands that the nuclei stay separate at any point in time, but does not restrict the image of any $\Gamma_j$ under $I$ artificially.

\subsection{Parameter Dependence of the Hartree–Fock Method} As described in Section \ref{disc_HF_Theory} any Hartree–Fock calculation starts with the choice of atomic orbitals $\mathfrak{A}$. This set typically depends on the nuclear positions, as in the case of Slater-type orbitals (STOs) or Gaussian-type orbitals (GTOs), where the orbitals are centered at the nuclei and decay exponentially with increasing distance from the centers. Therefore, we start the parameter dependent setting with the set of atomic orbitals
\begin{align}
    \mathfrak{A}(\omega)=\big\lbrace \chi_1(\omega),\ldots,\chi_{\text{n}_\text{b}}(\omega)\big\rbrace\subset \mathrm{H}^1(\mathbb{R}^3),\qquad(\forall\omega\in \Omega),
\end{align}
which are supposed to be linear independent for any nuclear configuration $\omega\in \Omega$. This gives an invertible overlap matrix function $\mathbf{S}:\Omega \longrightarrow \mathbb{R}^{\text{n}_\text{b}\times \text{n}_\text{b}}$, with $\mathbf{S}(\omega)$ being the overlap matrix described in \eqref{overlap} for any $\omega \in \Omega$. Moreover, we define the parameter dependent Hartree–Fock energy functional in the density perspective as

\begin{align}\label{param_ene_func}
    \mathcal{E}:\Omega\times \mathcal{G}\longrightarrow \mathbb{R}, \quad (\omega, \mathscr{D})\mapsto \mathcal{E}(\omega,\mathscr{D}):=\mathcal{E}_\omega(\mathscr{D}):= \text{Tr}((\mathbf{h}(\omega)+\tfrac{1}{2}\mathbf{G}(\omega,\mathscr{D}))\mathscr{D}).
\end{align}
Here $\mathbf{h}(\omega)$ and $\mathbf{G}(\omega,\mathscr{D})$ are, similar to the parameter dependent overlap matrix, the natural extension of the matrices defined in Section \ref{disc_HF_Theory} to the parameter dependent setting. We would like to stress that the parameter domain $\Omega$ describing the nuclear degrees of freedom can be decoupled from the Grassmannian $\mathcal{G}$ only because we have transformed the functional to the orthogonal orbital basis. Otherwise, the Grassmannian $\mathcal{G}(\omega)$ would itself be a manifold that depends on $\omega$ through variations in the overlap matrix function $\mathbf{S}(\omega)$. The corresponding parameter dependent version of the Riemannian gradient \eqref{riem_grad} is given by
\begin{align}\label{param_grad}
    \nabla_\mathcal{G}\mathcal{E}:\Omega\times \mathcal{G}\longrightarrow T\mathcal{G},\quad (\omega,\mathscr{D})\mapsto \nabla_\mathcal{G}\mathcal{E}(\omega,\mathscr{D}):=\nabla_\mathcal{G}\mathcal{E}_\omega (\mathscr{D}):=[\mathscr{D},[\mathscr{D},\mathbf{F}(\omega,\mathscr{D})]].
\end{align}
Our main objective in this section is, to investigate under what conditions we can expect to find parametrized densities  $\mathscr{D}:\Omega_0\subset \Omega\longrightarrow \mathcal{G}$, that are real analytic and satisfy
\begin{align*}
    \nabla_\mathcal{G}\mathcal{E}_\omega(\mathscr{D}(\omega))=0\in T_{\mathscr{D}(\omega)}\mathcal{G}, \qquad (\forall \omega\in \Omega_0).
\end{align*}
Recall that $\mathcal{G}$ admits a real analytic structure and therefore $\mathscr{D}(\omega)$ being analytic means, that for any chart $(\varphi,U)$ of that structure, the local map $\varphi \circ \mathscr{D}:\Omega_0 \longrightarrow \varphi(U) \subset \mathbb{R}^{\text{dim}\mathcal{G}}$ is real analytic in the usual sense.
The first condition we impose is natural and quite obvious.
\begin{framed}
\textbf{Analyticity Condition:} The Hartree–Fock energy functional $\mathcal{E}$ varies analytical in $\omega$.
\end{framed}In our particular case, where $\omega \in \Omega \subset \mathbb{R}^{3M}$ describes the position of the atoms, the parameter varies the electronic Hamiltonian $\mathcal{H}(\omega)$ by shifting the nuclei-electron interaction potential \eqref{nuclei_elec_pot} as $V(x,\omega)$. Moreover, in computational chemistry the vast majority of atomic orbitals $\mathfrak{A}(\omega)$, that are chosen for the discretization, are Gaussian-type orbitals (GTO) centered at the nuclei. To be more precise, often the atomic orbitals are either \textit{primitive Cartesian Gaussian-type} or some linear combination of those. A Cartesian Gaussian-type orbital, centered at nucleus $A\in \mathbb{R}^{3}$, with exponent $a>0$, and the Cartesian quantum numbers $i,j,k\in \mathbb{N}$ has the form
\begin{align}\label{GTO}
    \chi_{ijk}(x,A)=(x_1-A_1)^i(x_2-A_2)^j(x_3-A_3)^k\exp\big(-a\|x-A\|^2\big).
\end{align}
In contrast to the originally physical motivated Slater-type orbitals, the GTO's, allow for closed analytical expression for all the components of $\tilde{\mathbf{h}}, \tilde{\mathbf{G}}$ and $\mathbf{S}$ in the nuclear coordinates \cite{Boys}. All the necessary matrix components are expressed by means of the \textit{Boys function}
\begin{align}
   F:[0,\infty)\longrightarrow \mathbb{R}, \quad z\mapsto F(z):= \int^1_0 \mathrm{e}^{-zu^2}\mathrm{d}u,
\end{align} and its derivatives \cite{MCMURCHIE1978218,Obarasaika,Beylkin_2021}. For example, let us consider two \textit{spherical Gaussian orbitals} $\chi(x,A)$ and $\chi(x,B)$  centered at nucleus $A$ and $B$ respectively. Spherical Gaussian orbitals are just primitive Gaussian's with quantum numbers $i=j=k=0$ and hence $\chi(x,A)=\exp(-a\|x-A\|^2)$ and $\chi(x,B)=\exp(-b\|x-B\|^2)$ for some suitable fixed constants $a,b>0$. For the sake of simplicity, we will neglect any normalization constants. The product of the two spherical Gaussian's can then be written as a new spherical Gaussian 
\begin{align}
    \chi(x,A)\chi(x,B)= \exp(-a\|x-A\|^2)\exp(-b\|x-B\|^2)=K_{ab}\exp(-p\|x-P\|^2),
\end{align}
where $p:=a+b,$ and
\begin{align}
    P:=\frac{aA+bB}{p},\quad K_{ab}:=\exp\Big(\frac{-ab\|A-B\|^2}{p}\Big).
\end{align}
This observation is commonly known in the quantum chemistry literature as the \textit{Gaussian product rule}. Any Coulomb-interaction integral with respect to a third nucleus $C\in \mathbb{R}^{3M}$ can then be written as
\begin{align*}
    \int_{\mathbb{R}^3}\frac{\chi(x,A)\chi(x,B)}{\|x-C\|}\mathrm{d}x=K_{ab}\int_{\mathbb{R}^3}\frac{\exp(-p\|x-P\|^2)}{\|x-C\|}\mathrm{d}x= \frac{2\pi K_{ab}}{p}F\big(p\|P-C\|^2\big),
\end{align*}
which is an analytic function in the nuclear coordinates. All other components of $\tilde{\mathbf{h}}, \tilde{\mathbf{G}}$ and $\mathbf{S}$ are obtained in a similar manner and are therefore analytic in nuclear displacement. In order to ensure that the energy functional \eqref{param_ene_func} is analytic we also require that $\mathbf{S}^{1/2}$ and $\mathbf{S}^{-1/2}$ are analytic. Since building the square root of a matrix function relies on the analyticity of the eigenvalues, the necessary assumptions link directly to Kato's perturbation theory. In his classical works, it was pointed out that, under variations in several parameters, one cannot expect the eigenvalues to exhibit the same regularity as the matrix itself. An example of this issue is already provided in \cite{Kato:1966:PTL} for the two-dimensional case. The eigenvalues of the matrix function 
\begin{align*}
    B(x,y)=\begin{pmatrix}
x & y \\
y & -x 
\end{pmatrix}
\end{align*}
are given by $\lambda_{1,2}(x,y)=\pm \sqrt{x^2+y^2}$, which are not totally differentiable at the origin. Of course, this inconvenience only occurs at points where the eigenvalues cross. As for the overlap matrix $\mathbf{S}(\omega)$, degeneracies cannot be ruled out. For any rigorous result, one must either assume that no eigenvalue crossings occur or restrict the analysis to the one-dimensional setting. We adopt the latter approach, i.e. restrict ourselves to the case where the displacement of the nuclei is parametrized by a real analytic trajectory (see Definition \ref{trajectory_nuc_spac}). Therefore, we can make use of classical results on eigenvalue perturbation of Kato and Rellich \cite{Kato:1966:PTL,rellich1969perturbation}, specialized to our setting in \cite{Dieci_Luca}, as
\begin{theorem}\label{sqrt_ana}
    Let $B:\mathbb{R}\longrightarrow \mathbb{R}^{n\times n}$, symmetric and positive definite for all $\mu \in \mathbb{R}$. Then, the unique positive square root $B^{1/2}$ is analytic in $\mu$ as well.
\end{theorem}

 As a direct consequence of Theorem \ref{sqrt_ana} for a fixed analytic trajectory $\Gamma:I\longrightarrow\mathbb{R}^{3M}$ the composite overlap matrix function $\mathbf{S}^{1/2}(\mu):=\mathbf{S}^{1/2}(\Gamma(\mu))$ can be considered real analytic for the orbitals used in practice by the computational chemist community. Moreover, the fact that the inverse can be written as
\begin{align*}
    \mathbf{S}^{-1/2}(\mu)=\frac{1}{\text{det}\big(\mathbf{S}^{1/2}(\mu)\big)}\text{adj}\big(\mathbf{S}^{1/2}(\mu)\big)
\end{align*}
shows that $\mathbf{S}^{-1/2}(\mu)$ varies with the same regularity as the square root. All these considerations justify that the analyticity condition on the discretized energy functional $\mathcal{E}$ are usually met when the displacement is along an analytic curve. In order to establish the local regularity of solutions to the discretized Hartree–Fock equations, besides the analyticity condition imposed on the functional $\mathcal{E}$ we will need a non-degeneracy condition on the Riemannian Hessian of the functional, in the spirit of the typical implicit function theorem setting.  We now establish a version of the implicit function theorem tailored to Riemannian gradients of functionals defined on an analytic manifold. This is the exact situation of our Hartree–Fock  gradient \eqref{param_grad}.

\begin{theorem}\label{IFT}
    Let $(\mathcal{M},g)$ be a real analytic Riemannian manifold of dimension $n$ and $\mathcal{E}_{\omega}$ an analytic functional on $\mathcal{M}$, that analytically depends on $\omega\in \Omega\subset \mathbb{R}^d$. Suppose that $p_0\in \mathcal{M}$ is a critical point of $\mathcal{E}_{\omega_0}$, i.e. the Riemannian gradient $\nabla_\mathcal{M}   \mathcal{E}_{\omega_0}(p_0)=0\in T_{p_0}\mathcal{M}$. Moreover, suppose $\,\mathrm{Hess}_{p_0}\mathcal{E}_{\omega_0}$ is non-degenerate, then there exists a unique real analytic parametrization $\gamma :\Omega_0\subset \Omega\longrightarrow \mathcal{M}_0\subset \mathcal{M}$ with $(\omega_0,p_0)\in \Omega_0\times \mathcal{M}_0$ and
    \begin{align*}
        \nabla_\mathcal{M} \mathcal{E}_{\omega} (\gamma(\omega))=0\in T_{\gamma(\omega)}\mathcal{M}, \qquad (\forall {\omega}\in \Omega_0).
    \end{align*}
\end{theorem}

\begin{proof}
    Note that for any parameter $\omega\in \Omega$ by definition $g_{p}(\nabla_\mathcal{M}  \mathcal{E}_{\omega} (p),\xi)=\mathrm{d}\mathcal{E}_{\omega}(p)[\xi]$ for all $\xi \in T_{p}\mathcal{M}$, where $\mathrm{d}\mathcal{E}_{\omega}(p)$ is the differential of $\mathcal{E}_\omega$ at point $p$. It follows that $\nabla_\mathcal{M}  \mathcal{E}_{\omega}(p)=0\in T_{p}\mathcal{M}$ if and only if $\text{rank}_p(\mathcal{E}_{\omega})=0$. This is exactly the case if for given point $p\in \mathcal{M}$ one has $\mathrm{d}\mathcal{E}_\omega(p)[\xi]=0$ for all $\xi \in T_p\mathcal{M}$, or, since the rank of the differential is chart independent, $D(\mathcal{E}_\omega\circ \varphi^{-1})(\varphi(p))=0\in (\mathbb{R}^{n})^*$, where $(\varphi,U)$, with $\varphi:U\subset \mathcal{M}\longrightarrow \varphi(U)\subset \mathbb{R}^n$, is an arbitrary chart around $p$. Suppose that we fixed a chart $(\varphi, U)$ around the critical point $p_0\in \mathcal{M}$ for parameter $\omega_0\in \Omega$ and consider the map
    \begin{align}
        \mathcal{Y}:\Omega\times\varphi(U) \longrightarrow (\mathbb{R}^n)^*,
    \end{align}
    where $\mathcal{Y}(\omega,v):=D(\mathcal{E}_\omega \circ \varphi^{-1})(v)$ for any $(\omega,v)\in \Omega\times \varphi(U)$.
    By assumption we have for the point $(\omega_0,\varphi(p_0))$ that $\mathcal{Y}(\omega_0,\varphi(p_0))=D(\mathcal{E}_{\omega_0} \circ \varphi^{-1})(\varphi(p_0))=0\in (\mathbb{R}^n)^*$. We would like to apply the real analytic implicit function theorem at this point. Since $D\mathcal{Y}(\omega_0,\varphi(p_0))=D^2(\mathcal{E}_{\omega_0} \circ \varphi^{-1}))(\varphi(p_0))$, it is necessary that the second derivative $D^2(\mathcal{E}_{\omega_0} \circ \varphi^{-1}))(\varphi(p_0)):\mathbb{R}^n\longrightarrow (\mathbb{R}^n)^*$ is a linear isomorphism, where again $D$ denotes the differential operator with respect to the variables that are independent of the parameters. Note that if we pick $(\varphi,U)$ to be normal coordinates around $p_0$, i.e. the Christoffel symbols vanish at $p_0$, we get 
    \begin{align}
        D^2(\mathcal{E}_{\omega_0}\circ \varphi^{-1})(\varphi(p_0))[v, u] = \langle\text{Hess}_{\varphi(p_0)}(\mathcal{E}_{\omega_0}\circ \varphi^{-1})[v],u\rangle,\qquad(\forall v, u \in \mathbb{R}^n),
    \end{align}
    where $\langle\cdot\,,\cdot\rangle$ is the Euclidean inner product on $\mathbb{R}^n$.
    Since we assumed that $\text{Hess}_{p_0}\mathcal{E}_{\omega_0}$ is non-degenerated it follows that $D^2(\mathcal{E}_{\omega_0}\circ \varphi^{-1})(\varphi(p_0))$ is an isomorphism. Therefore, we can apply the real analytic implicit function theorem on $\mathcal{Y}$ and obtain an open neighborhood $\Omega_0 \subset \Omega$ of $\omega_0$ such that there is real analytic $c:\Omega_0\longrightarrow V\subset \varphi(U)\subset \mathbb{R}^n$, where $V$ is an open neighborhood of $\varphi(p_0)$, with $\mathcal{Y}(\omega,c(\omega))=D(\mathcal{E}_{\omega}\circ \varphi^{-1})(c(\omega))=0 \in (\mathbb{R}^n)^*$ for all $\omega\in \Omega_0$. We can use this map to define $\gamma:=\varphi^{-1}\circ c$. We remind that, because the manifold is assumed to carry a real-analytic structure, all coordinate charts, and their inverse maps, are real-analytic. Hence $\gamma: \Omega_0 \longrightarrow \varphi^{-1}(V)=:\mathcal{M}_0$ is real analytic, as well. Note, since along this map the differential $\mathrm{d}\mathcal{E}_{\omega}(\gamma(\omega))$ has constant rank zero it follows that for any $\omega\in \Omega_0$ we have
    \begin{align}
        g_{\gamma(\omega)}(\nabla_\mathcal{M}  \mathcal{E}_\omega(\gamma(\omega)),\xi)=\mathrm{d}\mathcal{E}_{\omega}(\gamma(\omega))[\xi]=0,\qquad\qquad(\forall \xi \in T_{\gamma(\omega)}\mathcal{M}),
    \end{align}
    and hence $\nabla_\mathcal{M}  \mathcal{E}_\omega(\gamma(\omega))=0\in T_{\gamma(\omega)}\mathcal{M}$ for any $\omega\in \Omega_0$. 
\end{proof}

The setting of Theorem \ref{IFT} establishes the natural conditions that we require for our Hartree–Fock framework. For the Hartree–Fock energy functional \eqref{param_ene_func}, $\text{Hess}_{\mathscr{D}}\mathcal{E}_{\omega}:T_{\mathscr{D}}\mathcal{G}\longrightarrow T_{\mathscr{D}}\mathcal{G}$ is known (see e.g. \cite{laura, gasp}) and at any critical point $(\omega_0,\mathscr{D}_0)$ has the form of 
\begin{align}
    \text{Hess}_{\mathscr{D}_0}\mathcal{E}_{\omega_0}(X)&=[\mathscr{D}_0,[X,\mathbf{F}(\omega_0,\mathscr{D}_0)]]+[\mathscr{D}_0,[\mathscr{D}_0, D^2 E(\omega_0,\mathscr{D}_0)X]],\qquad(\forall X\in T_{\mathscr{D}_0}\mathcal{G}),
\end{align}
where $E$ is the flat ambient energy functional \eqref{ambient energy functional}. 
\begin{corollary}\label{HF_coroll}
Suppose that the discretized Hartree–Fock energy functional $\mathcal{E}_\omega$ in the orthogonal orbital basis \eqref{param_ene_func} is analytic on $\mathcal{G}$ and depends analytically on $\omega\in \Omega \subset \mathbb{R}^{3M}$. Furthermore, suppose that $(\omega_0,\mathscr{D}_0)\in \Omega \times \mathcal{G}$ is a critical point of $\mathcal{E}$, i.e.
\begin{align}
    \nabla_\mathcal{G}\mathcal{E}_{\omega_0}(\mathscr{D}_0)=[\mathscr{D}_0,[\mathscr{D}_0,\mathbf{F}(\omega_0,\mathscr{D}_0)]]=0
\end{align}
such that the linear map $H:T_{\mathscr{D}_0}\mathcal{G}\longrightarrow T_{\mathscr{D}_0}\mathcal{G}$ with 
\begin{align}
    X\mapsto H(X):=[\mathscr{D}_0,[X,\mathbf{F}(\omega_0,\mathscr{D}_0)]]+[\mathscr{D}_0,[\mathscr{D}_0, D^2E(\omega_0,\mathscr{D}_0)X]],
\end{align}
is an isomorphism. Then there exists an open subset $\Omega_0\subset \Omega$ and a real analytic map $\mathscr{D}:\Omega_0 \longrightarrow \mathcal{G}$ such that
\begin{align}
    \nabla_\mathcal{G}\mathcal{E}_{\omega}(\mathscr{D}(\omega))=[\mathscr{D}(\omega),[\mathscr{D}(\omega),\mathbf{F}(\omega,\mathscr{D}(\omega))]]=0, \qquad(\forall \omega\in \Omega_0).
\end{align}
\end{corollary}

We want to stress again that, rigorously speaking, in the case where we parameterize the displacement of the nuclei via a real analytic curve $\Gamma:I\longrightarrow\mathbb{R}^{3M}$, the assumption on the regularity of the discretized Hartree–Fock functional is rather mild. This is a consequence of the chemist community using exclusively Gaussian-type orbitals and orbitals derived from them. For higher dimensional parameter variation, to the best of our knowledge, it is not known whether the square root of the overlap and its inverse retain analyticity as established in the one-dimensional setting by Theorem \ref{sqrt_ana}. Under similar restrictions, we can establish an analogous result for wavefunction-based post–Hartree–Fock methods. 

\begin{theorem}\label{analytic_C}
    Suppose $\Gamma:I\longrightarrow \mathbb{R}^{3M}$ is a real analytic curve in the nuclear coordinate space. Furthermore, let $\mathcal{E}_\mu:= \mathcal{E}(\Gamma(\mu),\mathscr{D})$ be the Hartree–Fock energy functional \eqref{param_ene_func} parametrized along the trajectory $\Gamma$. If $\mathcal{E}_\mu$ is an analytic functional on $\mathcal{G}$ and is real analytic in $\mu$ on $I$, $(\mu_0,\mathscr{D}_0)\in I\times \mathcal{G}$ is a critical point and the linear map
    $H:T_{\mathscr{D}_0}\mathcal{G}\longrightarrow T_{\mathscr{D}_0}\mathcal{G}$ with
    \begin{align}
        X\mapsto H(X):= [\mathscr{D}_0,[X,\mathbf{F}(\Gamma(\mu_0),\mathscr{D}_0)]]+[\mathscr{D}_0,[\mathscr{D}_0, D^2E(\Gamma(\mu_0),\mathscr{D}_0)X]],
    \end{align}
    is an isomorphism. Then there exist real analytic curves $C:I\longrightarrow \mathbb{R}^{\mathrm{n}_\mathrm{b}\times \mathrm{n}_\mathrm{b}}$ and $\Lambda:I\longrightarrow \mathbb{R}^{\mathrm{n}_\mathrm{b}\times \mathrm{n}_\mathrm{b}}$ with
    \begin{align}\label{param_dep_roothaan}
        \mathbf{F}(\mu,\mathscr{D}(\mu))C(\mu)=C(\mu)\Lambda(\mu), \quad C^\top(\mu) C(\mu)=\mathrm{I},\qquad(\forall \mu \in I).
    \end{align}
\end{theorem}
\begin{proof}
    By using the analytic curve located on $\mathcal{G}$ provided by Corollary \ref{HF_coroll}, one can construct an analytic Fock matrix. For this matrix, the classical Kato–Rellich theory applies \cite{rellich1969perturbation, Kato:1966:PTL}.
\end{proof}
Note that for the analytical curve $C:I\longrightarrow \mathbb{R}^{\text{n}_\text{b}\times \text{n}_\text{b}}$, provided by Theorem \ref{analytic_C}, to be properly ordered according to \eqref{occ_virt_C_matrix} it is necessary to have a positive HOMO-LUMO gap $\gamma>0$ across all of $I$, i.e.
\begin{align}
    \lambda_{N+1}(\mu)-\lambda_N(\mu)\ge \gamma >0, \qquad(\forall \mu\in I).
\end{align}
In this case $\mathscr{D}(\mu)=C_{\text{occ}}(\mu)C_{\text{occ}}(\mu)^\top$. We want to emphasize that Kato's theory on eigenvector perturbation only provides a choice for analytic eigenvectors. This implies that when degeneracies occur in either the occupied or virtual orbital energies, identifying an appropriate choice can become difficult. More often than not the orbital energies are completely non-degenerate and only show isolated crossings at certain critical geometries $\mu\in I$. When all orbital energies are well separated, the orthonormality constraint and the aufbau principle admit a unique set of molecular orbitals, with each orbital defined up to an independent choice of sign. Whenever a degeneracy occurs in one or more orbital energies at an isolated geometry $\mu\in I$, the resulting non-uniqueness can lead to index flips, which damage the regularity. All these issues will be discussed in a later section.\\
\newline
Now that we have developed when $\mathscr{D}(\omega)$ is an analytic function in $\omega$, and when also $C(\omega)$ can be chosen to vary analytically, we want to be precise under what circumstances the atomic orbitals $\mathfrak{A}(\omega)$ will provide us with analytic molecular orbitals $\mathfrak{M}(\omega)$ in the LCAO (linear combination of atomic orbitals) ansatz. In order to do so, we additionally assume that the atomic orbitals $\mathfrak{A}(\omega)$ are Banach space valued analytic functions in the parameter $\omega$, i.e. for any $i\le \text{n}_\text{b}$ the function
\begin{align}\label{Banach_anal_cond}
    \chi_i:\Omega\longrightarrow \mathrm{L}^2(\mathbb{R}^3)
\end{align}
is a real analytic Banach space valued function. Hence, assuming $C(\omega)$ is analytic, for any $i\le \text{n}_\text{b}$ the molecular orbital $\psi_i:\Omega\longrightarrow\mathrm{L}^2(\mathbb{R}^3)$ with
\begin{align}
    \psi_i(\omega)=\sum_{j=1}^{\text{n}_\text{b}}C_{ij}(\omega)\chi_j(\omega),\qquad (\forall \omega \in \Omega),
\end{align}
is a real analytic Banach space valued function, too. Note that condition \eqref{Banach_anal_cond} is satisfied in the case of GTO's \eqref{GTO}.

\subsection{Parameter Dependence of the Coupled Cluster Method}

In this section we will establish similar regularity results for the coupled cluster amplitudes $(t_\nu(\omega))_{\nu \in \mathscr{I}}\subset \mathbb{R}$. We claim that the analyticity assumptions, which are necessary for the coupled cluster method, are justified by the results of the previous section. Since we work with a set of parameter dependent molecular orbitals $\mathfrak{M}(\omega)$, obtained by the Hartree–Fock method, the Slater basis $\mathcal{S}(\mathfrak{M}(\omega))$, the excitation operators \eqref{exc_op} and the second quantized Hamiltonian \eqref{sec_quan_Hamil} also depend on $\omega$. We will denote the parameter dependent Hamiltonian by $\mathcal{H}_{\mathfrak{M}}(\omega):=\mathcal{H}_{\mathfrak{M}(\omega)}$.  The parameter-dependent coupled-cluster function, for any $\ell\le N$, is then defined as
\begin{align}\label{param_cc_func}
    \mathcal{Q}:\Omega\times \mathbb{R}^{|\mathscr{I}_\ell|}\longrightarrow \mathbb{R}^{|\mathscr{I}_\ell|}, \quad (\omega, t)\mapsto \mathcal{Q}(\omega,t),
\end{align}
where for any $\nu\in \mathscr{I}_\ell$ we have $\mathcal{Q}_\nu(\omega,t):=\langle \Psi_\nu(\omega),\mathrm{e}^{-T(\omega,t)}\mathcal{H}_\mathfrak{M}(\omega)\mathrm{e}^{T(\omega,t)}\Psi_0(\omega) \rangle_{\mathrm{L}^2}$. Here, $T$ is the cluster function
\begin{align}\label{cluster_function}
    T:\Omega\times\mathbb{R}^{|\mathscr{I}_\ell|}\longrightarrow \mathcal{B}(\mathfrak{F}(\mathfrak{h})), \quad (\omega, t)\mapsto \sum_{\nu\in \mathscr{I}_\ell}t_\nu \mathscr{X}_\nu(\omega).
\end{align}
Note that we restrict the function to real-valued amplitudes as input, since the solution amplitudes are known to be real.
\begin{lemma}\label{creation_annihili_analytic}
    Suppose $\mathfrak{M}(\omega)=\lbrace \psi_i(\omega)\rbrace_{i=1}^N\cup \lbrace \psi_a(\omega)\rbrace_{a=N+1}^{\mathrm{n}_\mathrm{b}}$ is set of real analytic molecular orbitals as Banach space-valued functions. Then the corresponding creation operators $\lbrace a^\dagger_i(\omega)\rbrace_{i=1}^N\cup \lbrace a^\dagger_a(\omega)\rbrace_{a=N+1}^{\mathrm{n}_\mathrm{b}} $ and annihilation operators $\lbrace a_i(\omega)\rbrace_{i=1}^N\cup \lbrace a_a(\omega)\rbrace_{a=N+1}^{\mathrm{n}_\mathrm{b}}$ are real analytic, as well. Moreover, any excitation operator $\mathscr{X}_\nu(\omega)$, for $\nu\in \mathscr{I}$, and the cluster function \eqref{cluster_function} are also real analytic. 
\end{lemma}

\begin{proof}
    It suffices to prove the analyticity for one creation operator $a^\dagger_i(\omega)=a^\dagger(\psi_i(\omega))$. Let $\omega_0\in\Omega$ be an arbitrary parameter and $U\subset \Omega$ an open neighborhood for which
    \begin{align*}
        \psi_i(\omega)=\sum_{\alpha\in \mathbb{N}^{3M}}\psi_{i,\alpha}(\omega-\omega_0)^\alpha,\qquad(\forall \omega\in U),
    \end{align*}
    where $\psi_{i,\alpha}\in \mathrm{L}^2(\mathbb{R}^3)$ for all $\alpha \in \mathbb{N}^{3M}$ and where convergence is measured in the $\mathrm{L}^2$-norm. Since the operator that assigns functions to creation operators $a^\dagger:\mathrm{L}^2(\mathbb{R}^3)\longrightarrow\mathcal{B}(\mathfrak{F}(\mathrm{L}^2(\mathbb{R}^3))$ is linear and  continuous, it follows that
    \begin{align}\label{series_expansion_creation}
        a^\dagger(\psi_i(\omega))&=a^\dagger\Big(\sum_{\alpha\in \mathbb{N}^{3M}}\psi_{i,\alpha}(\omega-\omega_0)^\alpha\Big)\\
        &=\sum_{\alpha\in \mathbb{N}^{3M}}a^\dagger(\psi_{i,\alpha})(\omega-\omega_0)^\alpha.
    \end{align}
    Due to $\|\psi\|_{\mathrm{L}^2}=\|a^\dagger(\psi)\|_{\mathcal{B}(\mathfrak{F}(\mathrm{L}^2(\mathbb{R}^3)))}$ for any $\psi \in \mathrm{L^2(\mathbb{R}^3})$, it follows that the convergence of the operator series is in the $\mathcal{B}(\mathfrak{F}(\mathrm{L}^2(\mathbb{R}^3))$ norm. Hence, $a^\dagger_i:\Omega\longrightarrow \mathcal{B}(\mathfrak{F}(\mathrm{L}^2(\mathbb{R}^3))$ is an analytic Banach space valued function. The exact same arguments can be made for the annihilation operators. Furthermore, any excitation operator is just the composition of creations and annihilations and, therefore, also real analytic. Finally, it is easy to see that the cluster function \eqref{cluster_function}, as a linear combination of excitations, is analytic, as well. 
\end{proof} As a consequence to Lemma \ref{creation_annihili_analytic}, given molecular orbitals $\mathfrak{M}(\omega)=\lbrace \psi_i(\omega)\rbrace_{i=1}^N\cup \lbrace \psi_a(\omega)\rbrace_{a=N+1}^{\mathrm{n}_\mathrm{b}}$ that are Banach space-valued analytic functions, the corresponding set of Slater determinants $\mathcal{S}(\mathfrak{M}(\omega))$ is also analytic as Banach space-valued functions. Moreover, due to Lemma \ref{creation_annihili_analytic} it is sufficient to have analytic molecular orbitals $\mathfrak{M}(\omega)$ to assure that the coupled cluster function \eqref{param_cc_func} consists of real analytic components, i.e. for any $\nu\in \mathscr{I}_\ell$
\begin{align}
    \mathcal{Q}_\nu:\Omega\times \mathbb{R}^{|\mathscr{I}_\ell|}\longrightarrow \mathbb{R},\quad (\omega,t)\mapsto \langle \Psi_\nu(\omega),\mathrm{e}^{-T(\omega,t)}\mathcal{H}_\mathfrak{M}(\omega)\mathrm{e}^{T(\omega,t)}\Psi_0(\omega) \rangle_{\mathrm{L}^2},
\end{align}
is a real analytic function. Besides the assumption on the regularity of $\mathcal{Q}$, we also need a non-degeneracy assumption.
\begin{definition}\label{non-degen_defi} Let $\ell\le N$ be a fixed excitation level, $\mathfrak{M}(\omega)= \lbrace \psi_i(\omega)\rbrace_{i=1}^N\cup \lbrace \psi_a(\omega)\rbrace_{a=N+1}^{\mathrm{n}_\mathrm{b}}$ be a set of molecular orbitals and $\mathcal{Q}$ the coupled cluster function with respect to the truncated Slater basis $\mathcal{S}_\ell(\mathfrak{M}(\omega))$ (see \eqref{trunc_slater_basis}). 
    A root $(\omega^*, t^*)$ of the parameter dependent coupled cluster function $\mathcal{Q}$ is a non-degenerate root if for $\mathfrak{V}_\ell:=\mathrm{span}_{\mathbb{C}}\,\mathcal{S}_\ell(\mathfrak{M}(\omega^*))\setminus  \lbrace \Psi_0(\omega^*) \rbrace$ 
    \begin{align}\label{non-degen-assum}
        E_{CC} \notin \sigma\big(P_{\mathfrak{V}_\ell}\mathrm{e}^{-T(\omega^*,t^*)}\mathcal{H}_{\mathfrak{M}}(\omega^*)\mathrm{e}^{T(\omega^*,t^*)}P_{\mathfrak{V}_\ell}\big),
    \end{align}
  where $E_{CC}$ is the coupled cluster energy at point $(\omega^*,t^*)$ (see \eqref{cc_energy}) and $P_{\mathfrak{V}_\ell}$ the orthogonal projector onto $\mathfrak{V}_\ell$.
\end{definition}

In order to better understand the condition \eqref{non-degen-assum}, first consider FCC, i.e. the case $\ell = N$. Here, $E_{CC}=E_\ell$ is a true eigenvalue of the discretized Hamiltonian and $\Psi_0(\omega_*)$ is a corresponding eigenvector of the transformed Hamiltonian. Therefore, condition \eqref{non-degen-assum} is equivalent to $E_\ell$ being a non-degenerate eigenvalue of $\mathcal{H}_{\mathfrak{M}}(\omega^*)$. This assumption has been used in several related works, including \cite{hassan2025analysissinglereferencecoupled,Csirik_part2}. For $\ell < N$, the coupled-cluster energy is not necessarily an eigenvalue of $\mathcal{H}_{\mathfrak M}(\omega^*)$. In this setting, non-degeneracy is determined solely by the spectrum of the similarity-transformed Hamiltonian on the space spanned by the excited determinants $\mathfrak{V}_\ell$. This is closely related to the study of excited states in \textit{equation-of-motion coupled-cluster} (EOM-CC) theory \cite{Shavitt_Bartlett_2009}. More specifically, condition \eqref{non-degen-assum} just excludes  zero excitation energies, which is typically the case when EOM-CC is applied in practice. The following theorem on the regularity was established in the master thesis \cite{beck2025interpolation}.
\begin{theorem}\label{analytic_amplitudes}

    Let $\Omega\subset \mathbb{R}^{3M}$ be an open subset of the nuclear coordinate space and $\ell\le N$ be a fixed excitation level. Furthermore, let  $\mathfrak{M}(\omega)= \lbrace \psi_i(\omega)\rbrace_{i=1}^N\cup \lbrace \psi_a(\omega)\rbrace_{a=N+1}^{\mathrm{n}_\mathrm{b}}$ be a set of real analytic Banach space valued molecular orbitals and $\mathcal{S}_\ell(\mathfrak{M}(\omega))$ the corresponding truncated set of Slater determinants. If $t^* \in \mathbb{R}^{|\mathscr{I}_\ell|}$ is the cluster amplitude for a certain nuclear geometry $\omega^*\in \Omega $ with 
    \begin{align}
        \mathcal{Q}(\omega^*,t^*)=0,
    \end{align}
    such that $(\omega^*,t^*)$ is non-degenerate in the sense of Definition  \ref{non-degen_defi}, then there exists an open set $V\subset \Omega$ with $\omega^* \in V$ and an open set $U \subset \mathbb{R}^{|\mathscr{I}_\ell|}$ with $t^*\in U$, such that there is a unique, real analytic function $t:V\longrightarrow U$ with 
    \begin{align}
        \big\lbrace (\omega,t(\omega)) \,|\,  \omega\in V\big\rbrace=\big\lbrace  (\omega,t)\in V\times U \,|\,  \mathcal{Q}(\omega,t)=0\big\rbrace.
    \end{align}
\end{theorem}

\begin{proof}
    By the regularity assumptions on the molecular orbitals and Lemma \ref{creation_annihili_analytic} the  coordinate functions $\mathcal{Q}_\upsilon(\omega,t)$ are real analytic for any $\upsilon\in \mathscr{I}_\ell$. Hence, in order to use the implicit function theorem for analytic functions, we only need to prove that the partial Jacobian 
    \begin{align}\label{jacobian_Q}
        \left[ \frac{\partial \mathcal{Q}_\upsilon (\omega^*,t^*)}{\partial t_\nu}\right]_{\upsilon,\nu=1}^{|\mathscr{I}_\ell|}
    \end{align}
    is a regular matrix. In order to prove this condition let us first calculate an arbitrary relevant partial derivative of a coordinate function, i.e.
    \begin{align*}
        \frac{\partial \mathcal{Q}_\upsilon }{\partial t_\nu}&= \frac{\partial \langle \Psi_\upsilon(\omega),\mathrm{e}^{-T(\omega,t)}\mathcal{H}_\mathfrak{M}(\omega)\mathrm{e}^{T(\omega,t)}\Psi_0(\omega)\rangle_{\mathrm{L}^2} }{\partial t_\nu}\\
        &=\big\langle \Psi_\upsilon(\omega),\big(-\mathscr{X}_\nu(\omega)\mathrm{e}^{-T(\omega,t)}\mathcal{H}_\mathfrak{M}(\omega)\mathrm{e}^{T(\omega,t)}+\mathrm{e}^{-T(\omega,t)}\mathcal{H}_\mathfrak{M}(\omega)\mathscr{X}_\nu(\omega)\mathrm{e}^{T(\omega,t)}\big)\Psi_0(\omega)\big\rangle_{\mathrm{L}^2}\\
        &=-\big\langle \Psi_\upsilon(\omega),\mathscr{X}_\nu(\omega)\mathrm{e}^{-T(\omega,t)}\mathcal{H}_\mathfrak{M}(\omega)\mathrm{e}^{T(\omega,t)}\Psi_0(\omega)\big\rangle_{\mathrm{L}^2}+\big\langle \Psi_\upsilon(\omega),\mathrm{e}^{-T(\omega,t)}\mathcal{H}_\mathfrak{M}(\omega)\mathrm{e}^{T(\omega,t)}\Psi_\nu(\omega)\big\rangle_{\mathrm{L}^2},        \end{align*}
        where we used that all excitation operators and cluster operators commute (see Proposition \ref{proposition_excitation_comm}). In the following for simplicity we will use the notation $T^*:=T(\omega^*,t^*)$. The first summand at the solution to the discretized and truncated coupled cluster equation reads
    \begin{align*}
        \langle \Psi_\upsilon(\omega^*),\mathscr{X}_\nu(\omega^*)\mathrm{e}^{-T^*}\mathcal{H}_\mathfrak{M}(\omega^*)\mathrm{e}^{T^*}\Psi_0(\omega^*)\rangle_{\mathrm{L}^2}&=\langle \mathscr{X}^{\dagger}_\nu(\omega^*)\Psi_\upsilon(\omega^*),\mathrm{e}^{-T^*}\mathcal{H}_\mathfrak{M}(\omega^*)\mathrm{e}^{T^*}\Psi_0(\omega^*)\rangle_{\mathrm{L}^2}.
        \intertext{Notice that $\mathscr{X}^{\dagger}_\nu(\omega^*)\Psi_\upsilon(\omega^*)=\mathscr{X}^\dagger_\nu(\omega^*)\mathscr{X}_\upsilon(\omega^*)\Psi_0(\omega^*)$. There are three cases to consider. Either the virtual and occupied indices of $\nu$ are contained within the occupied and virtual indices of $\upsilon$, in this case  $\mathscr{X}^\dagger_\nu(\omega^*)\mathscr{X}_\upsilon(\omega^*)=\mathscr{X}_\eta(\omega^*)$ for some excitation $\eta \in \mathscr{I}_\ell$, or $\nu=\upsilon$, in which case $\mathscr{X}^\dagger_\nu(\omega^*)\mathscr{X}_\nu(\omega^*)=\mathrm{I}$. Otherwise $\mathscr{X}^\dagger_\nu(\omega^*)\mathscr{X}_\upsilon(\omega^*)=0$. So we have}
        -\langle \Psi_\upsilon(\omega^*),\mathscr{X}_\nu(\omega^*)\mathrm{e}^{-T^*}\mathcal{H}_\mathfrak{M}(\omega^*)\mathrm{e}^{T^*}\Psi_0(\omega^*)\rangle_{\mathrm{L}^2}&=  \begin{cases}
        \,0, \\
-\langle \Psi_\eta(\omega^*),\mathrm{e}^{-T^*}\mathcal{H}_\mathfrak{M}(\omega^*)\mathrm{e}^{T^*}\Psi_0(\omega^*)\rangle_{\mathrm{L}^2}, \\
-\langle \Psi_0(\omega^*),\mathrm{e}^{-T^*}\mathcal{H}_\mathfrak{M}(\omega^*)\mathrm{e}^{T^*}\Psi_0(\omega^*)\rangle_{\mathrm{L}^2}.
\end{cases}
    \end{align*}
    Since $(\omega^*,t^*)$ is a solution, the second case also vanishes and the third case returns the (possibly truncated) coupled cluster energy $E_{\ell}$, i.e.
    \begin{align}
        \langle \Psi_\upsilon(\omega^*),-\mathscr{X}_\nu(\omega^*)\mathrm{e}^{-T^*}\mathcal{H}_\mathfrak{M}(\omega^*)\mathrm{e}^{T^*}\Psi_0(\omega^*)\rangle_{\mathrm{L}^2}=-E_{\ell}\delta_{\upsilon \nu}.
    \end{align}
    Inserting this into the derivative expression, we obtain
    \begin{align*}
        \frac{\partial \mathcal{Q}_\upsilon (\omega^*,t^*)}{\partial t_\nu}=\big\langle \Psi_\upsilon(\omega^*),\mathrm{e}^{-T^*}\mathcal{H}_\mathfrak{M}(\omega^*)\mathrm{e}^{T^*}\Psi_\nu(\omega^*)\big\rangle_{\mathrm{L}^2}-E_{\ell}\delta_{\upsilon \nu}.
    \end{align*}
So for the partial Jacobian at the solution $(\omega^*,t^*)$ we have  
\begin{align}
    \left[ \frac{\partial \mathcal{Q}_\upsilon (\omega^*,t^*)}{\partial t_\nu}\right]_{\upsilon,\nu=1}^{|\mathscr{I}_\ell|}=\mathbf{J}-E_{\ell}\cdot\mathrm{I},
\end{align}
    where $\mathbf{J}$ is the matrix with components $\mathbf{J}_{\upsilon \nu}=\big\langle \Psi_\upsilon(\omega^*),\mathrm{e}^{-T^*}\mathcal{H}_\mathfrak{M}(\omega^*)\mathrm{e}^{T^*}\Psi_\nu(\omega^*)\big\rangle_{\mathrm{L}^2}$. Therefore, the partial Jacobian of the coupled cluster function is invertible if \begin{align}\label{main_cc_det}
        \det(\mathbf{J}-E_{\ell}\cdot\mathrm{I})\ne 0.
    \end{align} 
     Since $\sigma(\mathbf{J})=\sigma(P_{\mathfrak{V}_\ell}\mathrm{e}^{-T^*}\mathcal{H}_{\mathfrak{M}}(\omega^*)\mathrm{e}^{T^*}P_{\mathfrak{V}_\ell})$, where $P_{\mathfrak{V}_\ell}$ is the orthogonal projector on the excited determinants  $\mathfrak{V}_\ell=\mathrm{span}_{\mathbb{C}}\,\mathcal{S}_\ell(\mathfrak{M}(\omega^*))\setminus  \lbrace \Psi_0(\omega^*) \rbrace$, the inequality \eqref{main_cc_det} holds by the non-degeneracy assumption.
\end{proof}

\section{Application: Amplitude Interpolation}\label{sec:interpol}
This section is devoted to numerically investigating the regularity results of the previous section through the use of interpolation. In doing so, we also propose a practical interpolation method for the coupled cluster amplitudes that directly benefits from the high regularity. Theorem \ref{analytic_amplitudes} suggests that, provided the molecular orbitals $\mathfrak{M}(\omega)$ are chosen in a regular manner, the amplitudes $(t_\nu(\omega))_{\nu\in \mathscr{I}}$ are likewise well behaved. For example, in the work of Schrader and Kvaal \cite{schrader_kvaal}, the authors introduced \textit{Procrustes orbitals} to guarantee an analytic set of molecular orbitals. The primary issue with using canonical molecular orbitals obtained from the Hartree–Fock method, ordered according to the Aufbau principle, is the occurrence of index reorderings, that result from orbital energy crossings. A second problem is that, even without index reorderings, the choice is not unique due to an inherent sign ambiguity.

\begin{figure}[H]
\centering{\includegraphics[width=3.5in]{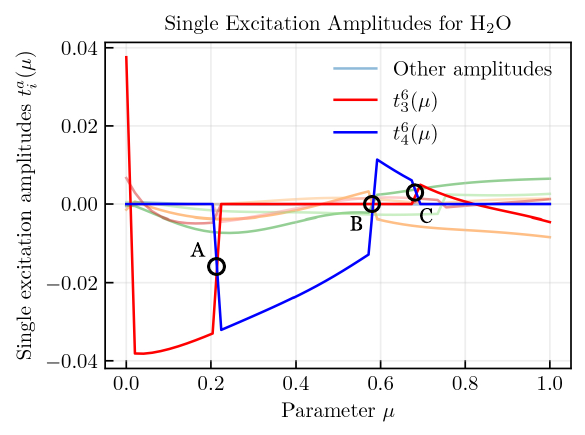}}
    \caption{The single excitation amplitudes $t^{a}_i$ of water in the STO-3G basis parameterized along a trajectory (see \eqref{trajectory_nuc_spac}) in the nuclear coordinate space. }
    \label{fig:MO_ampl}
\end{figure}

Figure \ref{fig:MO_ampl} shows the single excitation amplitudes of the water molecule in a minimal basis along a smooth trajectory in the nuclear coordinate space. Both obstacles that disrupt the regularity of the excitation amplitudes are visible. The highlighted amplitudes $t^6_3$ and $t^6_4$ show the effect of index reordering at points A and C. Moreover, point B shows a discontinuity in $t^6_4$ arising from a different choice of sign. We will make use of a basis transformation that removes these issues. Before proceeding, we introduce two general assumptions.

\begin{assumptionbox}{\textbf{Assumption I:}}\label{Assum_I}
Orbital energy degeneracies only appear for isolated geometries $\omega \in \Omega$.
\end{assumptionbox}
This assumption effectively decouples the discussion from the specific algorithm used to determine the coupled cluster amplitudes. Whenever there is a degeneracy in the orbital energies on an open set $U\subset \Omega$, there is a infinite number of possible choices for the coefficient matrices $C(\omega)$ for each parameter $\omega \in U$. Even if there exists a smooth choice over all of $U$, making that choice is up to the concrete software. Additionally, this occurs only for highly symmetric molecules, making it relatively predictable when it should be expected and when it should not.
\begin{assumptionbox}{\textbf{Assumption II:}}\label{Assum_II}
For all geometries $\omega\in \Omega$ there is a uniform HOMO-LUMO gap $\gamma>0$.
\end{assumptionbox}
The HOMO-LUMO gap assumption is necessary for the basis transformation to exist. Since the single reference coupled cluster method is poorly behaved when HOMO-LUMO crossings occur, it is a relatively mild assumption.\\
\newline
In order to proceed, it is necessary to observe that the parameter dependent Hamiltonian
\begin{align}
\mathcal{H}_\mathfrak{M}(\omega)=\sum_{p,q=1}^{\text{n}_\text{b}} h_{pq}(\omega)a^\dagger_p(\omega)a_q(\omega)+\frac{1}{2}\sum^{\text{n}_\text{b}}_{p,q,r,s=1}q^{rs}_{pq}(\omega)a^\dagger_p(\omega) a^\dagger_r(\omega)a_s(\omega)a_q(\omega)
\end{align}
 is invariant under index permutations. Therefore, even if index reorderings occur in the molecular orbitals, the Hamiltonian preserves its regularity in $\omega$. Moreover, consider for an arbitrary excitation index $\nu\in \mathscr{I}_\ell$ the parameter dependent excitation operator $\mathscr{X}_\nu(\omega)$, built by the corresponding molecular orbitals. Lemma \ref{creation_annihili_analytic} states, that this cluster operator is analytic if the molecular orbitals are. If there are index reorderings on the Hartree–Fock level of theory at isolated geometries $\omega \in \Omega$, then the excitation operator is analytic up to a reordering of the orbitals, i.e. there exists a parameter dependent permutation function \begin{align}\label{permutation_function}
\sigma:\Omega\times\mathscr{I}_\ell\longrightarrow \mathscr{I}_\ell
 \end{align} such that $\mathscr{X}_{\sigma(\omega,\nu)}(\omega)$ is analytic for all $\omega$ and all $\nu\in\mathscr{I}_{\ell}$. For the cluster function $T(\omega,t)$ (see  \eqref{cluster_function}), the permutation of the excitation indices is relevant, as well. The inconvenience arises when considering the parameter dependent coupled cluster solution $t:\Omega\longrightarrow \mathbb{R}^{|\mathscr{I}_\ell|}$. Given the parameter dependent solution in the amplitude space $(t_\nu(\omega))_{\nu\in\mathscr{I}_\ell}$, the solution cluster operator reads
 \begin{align}\label{solution_cluster_op_param}
     T(\omega,t(\omega))=\sum_{\nu \in \mathscr{I}_\ell} t_\nu(\omega)\mathscr{X}_\nu(\omega).
 \end{align}Due to the possible reordering of the molecular orbitals we can not apply Theorem \ref{analytic_amplitudes} here. Therefore, we consider the reordered coupled cluster function $\overline{\mathcal{Q}}$ with \begin{align}
     \overline{\mathcal{Q}}_{\nu}(\omega,t):=& \langle \Psi_{\sigma(\omega,\nu)}(\omega), \mathrm{e}^{-T_{\sigma(\omega)}(\omega,t)}\mathcal{H}_\mathfrak{M}(\omega)\mathrm{e}^{T_{\sigma(\omega)}(\omega,t)}\Psi_0(\omega)\rangle_{\mathrm{L}^2}\\
     =&\langle \mathscr{X}_{\sigma(\omega,\nu)}(\omega)\Psi_0(\omega), \mathrm{e}^{-T_{\sigma(\omega)}(\omega,t)}\mathcal{H}_\mathfrak{M}(\omega)\mathrm{e}^{T_{\sigma(\omega)}(\omega,t)}\Psi_0(\omega)\rangle_{\mathrm{L}^2}
 \end{align}
 where $\sigma$ is the permutation function \eqref{permutation_function} and $T_{\sigma(\omega)}(\omega,t):=\sum_{\nu \in \mathscr{I}_\ell} t_\nu\mathscr{X}_{\sigma(\omega, \nu)}(\omega)$. Note that by our  Assumption I and Assumption II the reference determinant $\Psi_0(\omega)$ is analytic up to possible phase flips. Since $\Psi_0(\omega)$ appears on both sides of the inner product, they cancel. The reordered coupled cluster function $\overline{\mathcal{Q}}$ fits the setting of Theorem \ref{analytic_amplitudes}, and with the non-degeneracy assumption, provides analytic solutions $(\bar{t}_\nu(\omega))_{\nu\in \mathscr{I}_\ell}$. Furthermore, by the relation
\begin{align}\label{smooth_ampl_equ_op}
    T_{\sigma(\omega)}(\omega, \bar{t}(\omega))=\sum_{\nu \in \mathscr{I}_\ell} \bar{t}_\nu(\omega)\mathscr{X}_{\sigma(\omega, \nu)}(\omega)=\sum_{\nu \in \mathscr{I}_\ell} \bar{t}_{\sigma(\omega,\nu)^{-1}}(\omega)\mathscr{X}_{\nu}(\omega)=T(\omega,(\bar{t}_{\sigma(\omega,\nu)^{-1}}(\omega))_{\nu\in \mathscr{I}_\ell}),
\end{align}
and the local uniqueness of the solution provided by the implicit function theorem, it follows that indeed 
\begin{align*}
    (t_{\nu}(\omega))_{\nu\in \mathscr{I}_\ell}=(\bar{t}_{\sigma(\omega,\nu)^{-1}}(\omega))_{\nu\in \mathscr{I}_\ell}.
\end{align*} So the amplitudes vary smoothly besides the multiplication with a suitable permutation matrix that directly corresponds with the index relabeling on the Hartree–Fock level of theory, i.e. in the coefficient matrix. The most obvious strategy to restore analyticity would be to attempt to obtain the permutation matrix that reorders the coupled cluster amplitudes. A reasonable approach is to employ a transformation that removes the irregularities caused by index relabelings without requiring any a priori analysis of the orbital energy crossing behavior.  We will show that the tensor transformation of the coupled cluster amplitudes from the molecular-orbital basis to the atomic-orbital basis corrects both index flips and any phase flips arising from inconvenient sign choices. Therefore, the transformed amplitudes retain the full analyticity that we established. Henceforth, we use the notation $\mathbb{O}:=\mathbb{R}^{\text{n}_\text{occ}}$ for the occupied space, $\mathbb{V}:=\mathbb{R}^{\text{n}_\text{virt}}$ for the virtual space and $\mathbb{B}:=\mathbb{R}^{\text{n}_\text{b}}$. The coupled cluster amplitudes can be stored in a tensor with size of the respective excitation rank. The rank $k\le N$ excitation amplitudes are stored in the tensor
\begin{align}\label{exc_tensor}
    \mathbf{T}_k:=\sum_{i_1,\ldots,i_k=1}^{\text{n}_\text{occ}}\sum^{\text{n}_\text{virt}}_{a_1,\ldots,a_{k}=1}t^{a_1,\ldots,a_k}_{i_1,\ldots,i_k} \big( \mathbf{v}_{a_1}\otimes \ldots \otimes \mathbf{v}_{a_k}  \big) \otimes \big(\mathbf{u}_{i_1}\otimes \ldots \otimes \mathbf{u}_{i_k}\big),
\end{align}
where $\lbrace \mathbf{v}_{a}\rbrace_{a=1}^{\text{n}_\text{virt}}$ is the standard basis of $\mathbb{V}$ and $\lbrace \mathbf{u}_{i}\rbrace_{i=1}^{\text{n}_\text{occ}}$ is the standard basis of $\mathbb{O}$. By construction $\mathbf{T}_k\in \mathbb{V}^{\otimes k}\otimes \mathbb{O}^{\otimes k}$. Theorem \ref{analytic_C} gives us, under certain assumptions, the existence of an analytic coefficient matrix in the nuclear displacement. The matter becomes algorithmically even more manageable if the orbital energies are completely separated, since the choice of the coefficient matrix as a function of the nuclear geometries is then unique up to the choice of signs in the columns. Furthermore, this convenience remains up to index reordering if one only considers isolated geometries where orbital energy crossings occur. Therefore, if we stay in that framework and consider the coefficient matrix function $C(\omega)$, there exists a permutation matrix function $\mathscr{P}:\Omega\longrightarrow \mathbb{R}^{\text{n}_\text{b}\times \text{n}_\text{b}}$ and integer valued functions $g_{\text{occ}},\, g_{\text{virt}}:\Omega \longrightarrow \mathbb{N}$ such that
\begin{align}\label{analytic_coeff_matrix_total}
    C(\omega)=(-1)^{g_{\text{occ}}(\omega)+g_{\text{virt}}(\omega)}\overline{C}(\omega)\mathscr{P}(\omega),
\end{align}
where $\overline{C}: \Omega\longrightarrow \mathbb{R}^{\text{n}_\text{b}\times \text{n}_\text{b}}$ is the analytic version of the coefficient matrix that solves the parameter dependent Roothaan equation \eqref{param_dep_roothaan}. Without loss of generality, we henceforth interpret $C(\omega)$ not as the coefficient matrix in an orthonormal orbital basis, but as the general (algorithmic) coefficient matrix function satisfying the $\mathbf{S}(\omega)$-orthonormality condition 
\begin{align}
    C^\top(\omega)\mathbf{S}(\omega)C(\omega)=\mathrm{I},
\end{align}
which is provided by most state-of-the-art quantum chemistry algorithms.
This choice does not limit the formulation, because any orthonormalized coefficient matrix $\tilde{C}(\omega)$ can be transformed through the relation $\tilde{C}(\omega)=\mathbf{S}^{1/2}(\omega)C(\omega)$.
For simplicity we will define $g(\omega):=g_{\text{occ}}(\omega)+g_{\text{virt}}(\omega)$. Moreover, if we assume that there is a uniform HOMO-LUMO gap $\gamma>0$ across all geometries $\omega\in \Omega$, then, leveraging the structure of \eqref{occ_virt_C_matrix},  $\mathscr{P}(\omega)$ has the block form consisting of two permutation matrices $\mathscr{P}_{\text{occ}}(\omega)\in \mathbb{R}^{\text{n}_\text{occ}\times \text{n}_\text{occ} }$ and $\mathscr{P}_{\text{virt}}(\omega)\in \mathbb{R}^{\text{n}_\text{virt}\times \text{n}_\text{virt} }$, i.e. 
\begin{align}\label{permutation_matrix_func_block}
\mathscr{P}(\omega) = \left( \begin{array}{c|c}
  \mathscr{P}_{\text{occ}}(\omega) & \mathbf{0} \\
  \hline
  \mathbf{0} & \mathscr{P}_{\text{virt}}(\omega)
\end{array} \right).
\end{align}
This yields the existence of two analytic matrix valued functions $\overline{C}_{\text{occ}}:\Omega\longrightarrow \mathbb{R}^{\text{n}_\text{b}\times \text{n}_\text{occ}}$ and $\overline{C}_{\text{virt}}:\Omega\longrightarrow \mathbb{R}^{\text{n}_\text{b}\times \text{n}_\text{virt}}$, such that for any $\omega \in \Omega$
\begin{align}
    C_{\text{occ}}(\omega)=(-1)^{g_{\text{occ}}(\omega)}\overline{C}_{\text{occ}}(\omega)\mathscr{P}_{\text{occ}}(\omega),\quad \text{and}\quad C_{\text{virt}}(\omega)=(-1)^{g_{\text{virt}}(\omega)}\overline{C}_{\text{virt}}(\omega)\mathscr{P}_{\text{virt}}(\omega).
\end{align}
This decomposability is quite important for the tensor transformation to cure any causes of discontinuities that can arise. Note that for any excitation rank $k\le N$, the excitation tensor $\mathbf{T}_k$ will vary in $\omega$ according to
\begin{align}\label{param_dep_k_fold_ex_tens}
    \mathbf{T}_k(\omega)&=\sum_{\substack{i_1,\ldots,i_k\le\text{n}_\text{occ}\\a_1,\ldots,a_{k}\le \text{n}_\text{virt}}} t^{a_1,\ldots,a_k}_{i_1,\ldots,i_k}(\omega) \big( \mathbf{v}_{a_1}\otimes \ldots \otimes \mathbf{v}_{a_k}  \big) \otimes \big(\mathbf{u}_{i_1}\otimes \ldots \otimes \mathbf{u}_{i_k}\big)\\
    &=(-1)^{kg(\omega)}\hspace{-8pt}\sum_{\substack{i_1,\ldots,i_k\le\text{n}_\text{occ}\\a_1,\ldots,a_{k}\le \text{n}_\text{virt}}} \bar{t}^{\;a_1,\ldots,a_k}_{\;i_1,\ldots,i_k}(\omega) \left(\bigotimes^{k}_{p=1}\mathscr{P}^*_{\text{virt}}(\omega)\mathbf{v}_{a_p}   \right) \otimes \left(\bigotimes^{k}_{q=1}\mathscr{P}^*_{\text{occ}}(\omega)\mathbf{u}_{i_q}   \right).
\end{align}
Hence, starting with any reference geometry $\omega_0\in\Omega$ and the corresponding tensor $\mathbf{T}_k(\omega_0)$, built with respect to a fixed frame $[\mathbf{v}_1,\ldots,\mathbf{v}_{\text{n}_{\text{virt}}}]$ and $[\mathbf{u}_1,\ldots,\mathbf{u}_{\text{n}_{\text{occ}}}]$,  the coupled cluster amplitudes in the tensor $\mathbf{T}_k(\omega)$ evolve independent of the reordering in the rows of the coupled cluster function, but the chosen frame gets reordered depending on $\omega$ and the corresponding row permutation in $\mathcal{Q}(\omega,t)$, due to the crossings on the Hartree–Fock level of theory. The rank $k$ excitation tensor \eqref{param_dep_k_fold_ex_tens} can be compactly written as
\begin{align}\label{param_exc_decomp}
    \mathbf{T}_k(\omega)=(-1)^{kg(\omega)} \big((\mathscr{P}_{\text{virt}}^*(\omega))^{\otimes k}\otimes (\mathscr{P}^*_{\text{occ}}(\omega))^{\otimes k}\big) \overline{\mathbf{T}}_k,
\end{align}
where
\begin{align}
    \overline{\mathbf{T}}_k(\omega):=\sum_{i_1,\ldots,i_k=1}^{\text{n}_\text{occ}}\sum^{\text{n}_\text{virt}}_{a_1,\ldots,a_{k}=1}\bar{t}^{\;a_1,\ldots,a_k}_{\;i_1,\ldots,i_k}(\omega) \big( \mathbf{v}_{a_1}\otimes \ldots \otimes \mathbf{v}_{a_k}  \big) \otimes \big(\mathbf{u}_{i_1}\otimes \ldots \otimes \mathbf{u}_{i_k}\big).
\end{align}
Recall \eqref{smooth_ampl_equ_op}, i.e. that the coupled cluster amplitudes appearing in $\overline{\mathbf{T}}_k(\omega)$ are smooth functions in $\omega$, hence $\overline{\mathbf{T}}_k:\Omega\longrightarrow \mathbb{V}^{\otimes k}\otimes \mathbb{O}^{\otimes k}$ carries the same regularity as the coefficient matrix $\overline{C}(\omega)$. The tensor transformation from molecular orbital basis to atomic orbital basis for excitation rank $k\le N$ is given by
\begin{align}\label{MO_to_AO}
\mathcal{A}_k:\mathbb{V}^{\otimes k}\otimes \mathbb{O}^{\otimes k} \longrightarrow \mathbb{B}^{\otimes k}\otimes \mathbb{B}^{\otimes k},\quad \mathbf{T}\mapsto
    \big(C_{\text{virt}}^{\otimes k}\otimes C_{\text{occ}}^{\otimes k}\big)\mathbf{T}.
\end{align} 
Note that the transformation \eqref{MO_to_AO} remains well-defined for alternative partitions of the coefficient matrix \eqref{occ_virt_C_matrix}, provided a spectral gap exists between the eigenvalues at the point of splitting. Using the orthonormality constraint $C^*(\omega)S(\omega)C(\omega)=\mathrm{I}$, for any tensor transformation \eqref{MO_to_AO} there is a left-inverse transformation 
\begin{align}\label{AO_to_MO}
    \mathcal{A}_k^{-1}:\mathbb{B}^{\otimes k}\otimes \mathbb{B}^{\otimes k} \longrightarrow \mathbb{V}^{\otimes k}\otimes \mathbb{O}^{\otimes k},\quad \mathbf{T}\mapsto \big((C^*_{\text{virt}}\mathbf{S})^{\otimes k}\otimes (C^*_{\text{occ}}\mathbf{S})^{\otimes k}\big)\mathbf{T}.
\end{align}
\begin{theorem}\label{MO_AO_makes_high_reg}
    Let $\Omega\subset \mathbb{R}^{3M}$ be subdomain of the nuclear coordinate space and $\mathfrak{A}(\omega)$ a set of atomic orbitals for all $\omega\in\Omega$, that provide a uniform HOMO-LUMO gap $\gamma>0$ for the discrete Hamiltonian $\mathcal{H}_\mathfrak{M}(\omega)$ of interest. Furthermore, assuming the orbital energy degeneracies occur only at isolated points, the regularity of the transformed tensor $(\mathcal{A}_k\mathbf{T}_k)(\omega)=\mathcal{A}_k(\omega)\mathbf{T}_k(\omega)$ for any excitation rank $k\le N$ is identical to that of $\overline{C}(\omega)$ . 
\end{theorem}
\begin{proof}
    The preceding analysis of the parameter-dependent excitation tensor $\mathbf{T}_k(\omega)$ makes the proof straightforward. Using \eqref{param_exc_decomp}
    \begin{align*}
        (\mathcal{A}_k\mathbf{T}_k)(\omega)&=\big(C_{\text{virt}}^{\otimes k}(\omega)\otimes C_{\text{occ}}^{\otimes k}(\omega)\big)\mathbf{T}_k(\omega)\\
        &=\big(C_{\text{virt}}^{\otimes k}(\omega)\otimes C_{\text{occ}}^{\otimes k}(\omega)\big)(-1)^{kg(\omega)} \big((\mathscr{P}_{\text{virt}}^*(\omega))^{\otimes k}\otimes(\mathscr{P}_{\text{occ}}^*(\omega))^{\otimes k}\big) \overline{\mathbf{T}}_k(\omega)\\
        &=(-1)^{kg(\omega)}\big((C_{\text{virt}}(\omega)\mathscr{P}_{\text{virt}}^*(\omega))^{\otimes k}\otimes (C_{\text{occ}}(\omega)\mathscr{P}_{\text{occ}}^*(\omega))^{\otimes k}\big) \overline{\mathbf{T}}_k(\omega)\\
        &=(-1)^{2kg(\omega)}\! \big((\overline{C}_{\text{virt}}(\omega)(\mathscr{P}_{\text{virt}}\mathscr{P}_{\text{virt}}^*)(\omega))^{\otimes k}\otimes(\overline{C}_{\text{occ}}(\omega)(\mathscr{P}_{\text{occ}}\mathscr{P}_{\text{occ}}^*)(\omega))^{\otimes k}\big) \overline{\mathbf{T}}_k(\omega)\\
        &=\big(\overline{C}^{\otimes k}_{\text{virt}}(\omega)\otimes \overline{C}^{\otimes k}_{\text{occ}}(\omega)\big) \overline{\mathbf{T}}_k(\omega).
    \end{align*}
    The assertion follows directly, as $\overline{\mathbf{T}}_k$ has the regularity properties of $\overline{C}$.
\end{proof}

\subsection{Excitation Tensor Interpolation}\label{exc_tens_inter} Suppose one is interested in solving the electronic structure problem for a large number of nuclear geometries, for instance, to find an appropriate wavefunction or to calculate the corresponding energy and forces. The naive approach is simply interpolating the raw amplitudes for a sample set of geometries. As previously discussed, this approach is destined to fail due to the reorderings and phase flips caused by the arbitrary algorithmic choices made during the calculation. Therefore, as a strategy, which does not require any a priori knowledge of the orbital crossing behavior, we suggest to leverage Theorem \ref{MO_AO_makes_high_reg} in order to bypass the irregularities of the molecular orbital basis by transforming into the atomic orbital basis. Henceforth, we only consider the case, where the nuclear displacement is provided by a trajectory (see Definition \ref{trajectory_nuc_spac}), i.e. we consider a path
\begin{align}
    \Gamma:I\longrightarrow \mathbb{R}^{3M}.
\end{align}
This is the natural setting for Theorem \ref{analytic_C}. Hence, by choosing $\Gamma$ as an analytic trajectory, we can expect that the reordered coefficient matrix $\overline{C}(\mu)$, of the canonical orbitals, varies analytically in $\mu$, for any implementation of the Roothaan algorithm. The interpolation method via the molecular orbital to atomic orbital basis transformation on the excitation tensors splits into an offline and an online stage. In practice, to effectively control the prefactors in memory and time without compromising asymptotic efficiency, the separation of stages must be handled more strategically than simply storing the transformed tensors during the offline phase. Further details regarding these refactoring subtleties, as well as the complexity analysis, are discussed later. 
\subsubsection{Offline Stage of the Interpolation}\label{offline_stage} In the offline stage of the interpolation we fix a set of sample points $\mathcal{N}:=\lbrace\mu_1,\ldots,\mu_d\rbrace$. For this set, an ordinary coupled cluster calculation of excitation level $\ell\le N$ is performed and provides the sets of excitation tensors
\begin{align}\label{offline_exc_ten}
    \big\lbrace\mathbf{T}_k(\mu_i)\big\rbrace_{k=1}^{\ell},\qquad (\forall i\le d).
\end{align}
Furthermore, the coefficient matrices provided by the Hartree–Fock calculations 
\begin{align}
    \big\lbrace C(\mu_1),\ldots, C(\mu_d)\big\rbrace,
\end{align}
are saved.
\subsubsection{Online Stage of the Interpolation}\label{online_stage} Given a $\mu\in I$ the approximation of excitation tensor $\mathbf{T}_k(\mu)$ of rank $k\le \ell$ is given by 
\begin{align}\label{tensor_interpolant}
    \widetilde{\mathbf{T}}_k(\mu):=\sum_{j=1}^d\underbrace{\mathcal{A}_k^{-1}(\mu)\mathcal{A}_k(\mu_j)}_{=:\, \mathscr{L}_k(\mu,\mu_j)}\mathbf{T}_k(\mu_j)\mathrm{L}_j (\mu)=\sum^{d}_{j=1} \mathscr{L}_k(\mu,\mu_j)\mathbf{T}_k(\mu_j)\mathrm{L}_j(\mu),
\end{align}
where $\mathrm{L}_j$ are the Lagrange polynomials with respect to $\mathcal{N}$. The major difference to polynomial interpolation on the excitation tensor \eqref{offline_exc_ten}, is that, we transform the tensor into the smoother version in the atomic orbital basis, and return them, after polynomial interpolation, to the molecular basis via applying the left inverse transform at the current parameter $\mu\in I$. Note that for any tensor $\mathbf{T}\in \mathbb{V}^{\otimes k}\otimes \mathbb{O}^{\otimes k}$ one has
\begin{align}\label{B_transform}
    \mathscr{L}_k(\mu,\mu_j)\mathbf{T}=\big(\mathscr{L}_{\text{virt}}^{\otimes k}(\mu,\mu_j)\otimes \mathscr{L}_{\text{occ}}^{\otimes k}(\mu,\mu_j)\big)\mathbf{T},
\end{align}
with the virtual transformation $\mathscr{L}_{\text{virt}}(\mu,\mu_j):=C^*_{\text{virt}}(\mu)\mathbf{S}(\mu)C_{\text{virt}}(\mu_j)\in \mathbb{R}^{\text{n}_{\text{virt}}\times \text{n}_{\text{virt}}}$ and the occupied transformation $\mathscr{L}_{\text{occ}}(\mu,\mu_j):=C^*_{\text{occ}}(\mu)\mathbf{S}(\mu)C_{\text{occ}}(\mu_j) \in \mathbb{R}^{\text{n}_{\text{occ}}\times \text{n}_{\text{occ}}}$.\\
\newline
Since the transformations \eqref{MO_to_AO} and \eqref{AO_to_MO} are linear, the approximation error can be written as
\begin{align*}
    \|\widetilde{\mathbf{T}}_k(\mu)-\mathbf{T}_k(\mu)\|_F&=\Big\| \sum^{d}_{\ell=1}\mathcal{A}_k^{-1}(\mu)(\mathcal{A}_k\mathbf{T}_k)(\mu_\ell)\mathrm{L}_\ell(\mu)-\mathcal{A}_k^{-1}(\mu)(\mathcal{A}_k\mathbf{T}_k)(\mu)\Big\|_F\\
    &\le\|\mathcal{A}_k^{-1}(\mu)\|_2 \underbrace{\Big\| \sum^{d}_{\ell=1}(\mathcal{A}_k\mathbf{T}_k)(\mu_\ell)\mathrm{L}_\ell(\mu)-(\mathcal{A}_k\mathbf{T}_k)(\mu)\Big\|_F}_{=:\,\mathbf{e}^\mu_k(\mathcal{N})}=\|\mathcal{A}_k^{-1}(\mu)\|_2\,\mathbf{e}^{\mu}_k(\mathcal{N}),
\end{align*}
where $\mathbf{e}^\mu_k(\mathcal{N})$ denotes the polynomial interpolation error of the excitation tensor in the atomic orbital basis with respect to the interpolation nodes $\mathcal{N}$. Furthermore, when choosing the Frobenius norm, the orthogonality relation $\big(C^*(\mu)\mathbf{S}^{1/2}(\mu)\big)\big(\mathbf{S}^{1/2}(\mu)C(\mu)\big)=\mathrm{I}$ provides the error estimates of the back transform
\begin{align*}
    \|\mathcal{A}_k^{-1}(\mu)\|_2&=\|C^*_{\text{virt}}(\mu)\mathbf{S}(\mu)\|_2^k\|C^*_{\text{occ}}(\mu)\mathbf{S}(\mu)\|_2^k\\
    &=\|C^*_{\text{virt}}(\mu)\mathbf{S}^{1/2}(\mu)\mathbf{S}^{1/2}(\mu)\|_2^k\|C^*_{\text{occ}}(\mu)\mathbf{S}^{1/2}(\mu)\mathbf{S}^{1/2}(\mu)\|_2^k\\
    &\le \|\mathbf{S}^{1/2}(\mu)\|_2^{2k}\|C^*_{\text{virt}}(\mu)\mathbf{S}^{1/2}(\mu)\|_2^k\|C^*_{\text{occ}}(\mu)\mathbf{S}^{1/2}(\mu)\|_2^k\\
    &= \|\mathbf{S}^{1/2}(\mu)\|_2^{2k}.
\end{align*}
Combining this estimate with the previous one gives the total estimate for the excitation tensor $\mathbf{T}_k$ of rank $k\le \ell$ at point $\mu$ via
\begin{align}
    \|\widetilde{\mathbf{T}}_k(\mu)-\mathbf{T}_k(\mu)\|_F\le \|\mathbf{S}(\mu)\|_2^{k}\mathbf{e}^{\mu}_k(\mathcal{N}).
\end{align}
The interpolation error $\mathbf{e}^\mu_k(\mathcal{N})$ is governed by the principles of classical approximation theory for polynomial interpolation (see, e.g. \cite{agarwal1993error}).

\subsection{Complexity Analysis of the Interpolation}

The interpolation of the excitation tensors $\mathbf{T}_k(\mu)$ boils down to interpolating the amplitudes in the atomic orbital basis, which are the coefficients of $(\mathcal{A}_k\mathbf{T}_k)(\mu)$, and transforming them back in the online phase. In the context of asymptotic analysis, the system size $N$ and the number of snapshots $d$ remain the only relevant parameters. Since $\text{n}_{\text{b}}$, $\text{n}_{\text{occ}}$ and $\text{n}_{\text{virt}}$ all scale linearly with the system size $N$, they can be treated equivalently. Nevertheless, we will split the considerations into each of the three quantities, as they possess significantly different prefactors, that can be decisive in practical computations. As the level of approximation gets closer to chemical accuracy, we usually have $\text{n}_{\text{b}} \approx \text{n}_{\text{virt}}$ and therefore $\text{n}_{\text{occ}}\ll \text{n}_{\text{virt}}$. For all $k\le \ell$ we have $(\mathcal{A}_k\mathbf{T}_k)(\mu)\in \mathbb{B}^{\otimes k}\otimes \mathbb{B}^{\otimes k}$, so the memory scaling in the offline stage would be $\mathcal{O}(\text{n}_{\text{b}}^{2k} d)$. Therefore, approximating the whole coupled cluster wavefunction of excitation level $\ell \le N$ scales in offline memory according to $\mathcal{O}(\text{n}_{\text{b}}^{2\ell} d)$. This memory scaling is already considered intractable, even for CCSD, where $\ell=2$. The structure of $\mathcal{A}_k$ and $\mathcal{A}_k^{-1}$ allows this issue to be avoided by working with $\mathscr{L}_k(\mu,\mu_j)$ directly (see \eqref{B_transform}), i.e. shifting all required contraction operations to the online stage. Building $\mathscr{L}_{\text{virt}}(\mu,\mu_j)$ and $\mathscr{L}_{\text{occ}}(\mu,\mu_j)$ in the online stage are cheap matrix multiplications, that can be neglected when thinking about the total computational scaling. The entire transformation $\mathscr{L}_k(\mu,\mu_j)$ can be performed by applying step by step  the matrices $\mathscr{L}_{\text{virt}}(\mu,\mu_j)$ and $\mathscr{L}_{\text{occ}}(\mu,\mu_j)$ from left and right, respectively, on each factor of the tensor. For example, applying a virtual transformation $\mathscr{L}_{\text{virt}}(\mu,\mu_j)$ from the left on the $i\le k$ virtual index of a tensor $\mathbf{T}\in \mathbb{V}^{\otimes k}\otimes \mathbb{O}^{\otimes k}$ requires the calculation of the \textit{n-mode product} $\mathbf{T}\times_i \mathscr{L}_{\text{virt}}(\mu,\mu_j)$ (see e.g. \cite{tensor_ana}), where the components of the new tensor are
\begin{align}\label{n_way_tensor}
    \big(\mathbf{T}\times_i \mathscr{L}_{\text{virt}}(\mu,\mu_j)\big)^{a_1\ldots a_{i-1}s a_{i+1}\ldots a_k}_{i_1\ldots i_k}= \sum^{\text{n}_{\text{virt}}}_{a_i=1}t^{a_1 \ldots a_k}_{i_1\ldots i_k}\big(\mathscr{L}_{\text{virt}}(\mu,\mu_j)\big)_{s a_i}.
\end{align}
Calculating all new components \eqref{n_way_tensor} has a scaling of time complexity of $\mathcal{O}(\text{n}_{\text{virt}}^{k+1} \text{n}_\text{occ}^k)$. Since $ \text{n}_{\text{occ}} \ll \text{n}_{\text{virt}}$, it is sufficient to count only the virtual transformations, i.e. contractions with  $\mathscr{L}_{\text{virt}}(\mu,\mu_j)\in \mathbb{R}^{\text{n}_\text{virt}\times \text{n}_\text{virt}}$ which leads to a total online time scaling of $\mathcal{O}(\text{n}_{\text{virt}}^{k+1} \text{n}_\text{occ}^{k}k)$ to calculate the entire $\mathscr{L}_k(\mu,\mu_j)\mathbf{T}_k(\mu_j)$, and hence $\mathcal{O}(\text{n}_{\text{virt}}^{k+1} \text{n}_\text{occ}^{k} kd  )$ to build the entire interpolant \eqref{tensor_interpolant}. Notably, the memory at no point exceeds the scaling of $\mathcal{O}( \text{n}_{\text{virt}}^\ell \text{n}_\text{occ}^\ell d )$. The time complexity in the offline phase of course scales as the respective exact coupled cluster method with excitation level $\ell \le N$ multiplied by the number of snapshots, which is $\mathcal{O}(\text{n}_{\text{virt}}^{\ell+2}\text{n}^\ell_\text{occ}d)$ or generally $\mathcal{O}(N^{2\ell+2}d)$. However, in the online phase of the interpolation the time scaling will be $\mathcal{O}(N^{2\ell+1}d)$ and therefore reducing the scaling exponent by one. Due to $\mathbf{T}_k(\mu)\in \mathbb{V}^{\otimes k}\otimes \mathbb{O}^{\otimes k}$ for any $k\le \ell $, one can not expect to go below the time scaling $\mathcal{O}(N^{2\ell}d)$ for any direct interpolation method, that approximates the whole coupled cluster wavefunction of excitation level $\ell\le N$. Surpassing this lower bound likely requires the preliminary use of a low-rank approximation, or perhaps some similar method of tensor compression.

\subsection{Numerical Experiments}\label{NumExp}

In this section, we perform numerical experiments for CCSD, corresponding to excitation level $\ell=2$ in order to illustrate the interpolation method.
When it comes to polynomial-like interpolation methods, interpolation along the \textit{Chebyshev nodes} offers good experimental insight into the regularity of the approximated object. It is well known that Chebyshev interpolation of real analytic functions exhibits exponential error decay as the number of nodes increases, with the rate depending on how far the function can be extended holomorphically into the complex plane \cite{mason2002chebyshev,cheby_tad}. The following numerical experiments are inspired by Born–Oppenheimer molecular dynamics \cite{normal_modes}. In this approach, the time evolution of a molecule is approximated using classical mechanics. Instead of computing the exact quantum mechanical dynamics, one solves Newton’s equations of motion near an equilibrium configuration to obtain an approximate description of the system’s behavior over time. This time evolution is then expressed in terms of the eigenvectors of the Hessian of the energy functional, the \textit{normal modes}, and their corresponding eigenvalues, which are the \textit{normal mode frequencies}. For each numerical experiment a real analytic trajectory $\Gamma$ in the nuclear coordinate space was prepared in the form
\begin{align}\label{experiment_interpolant}
    \Gamma(\mu) = \Gamma_0+\sum^{m}_{s=1} c_s \sin(2\pi \omega_s \mu) \zeta_s.
\end{align}
Here $c_1,\ldots ,c_m\in \mathbb{R}$ are some arbitrary coefficients and $\Gamma_0$ is an equilibrium configuration of given molecule, determined by a geometry optimization on the DFT level and run via \texttt{PySCF} \cite{PYSCF,Pyscf_geom}. Moreover, $\zeta_1,\ldots,\zeta_m \in \mathbb{R}^{3M}$ are some of the normal modes and $\omega_1,\ldots, \omega_m \in \mathbb{R}$ their respective normal mode frequencies. We always consider the parameter domain, or time domain $I=[0,1]$. Therefore, we pick Chebyshev nodes of the first kind $\mathcal{N}_d=\lbrace \mu_0,\ldots, \mu_d \rbrace$ with respect to the unit interval.
\begin{figure}[H]
\centering{\includegraphics[width=6in]{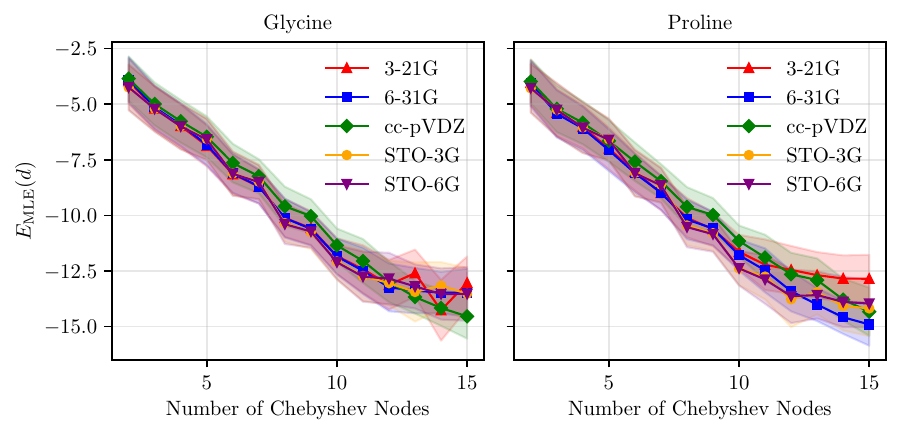}}
    \caption{The $E_{\text{MLE}}(d)$  for the amino acids glycine $\mathrm{C}_2\mathrm{H}_5\mathrm{N}\mathrm{O}_2$ (left) and proline $\mathrm{C}_5\mathrm{H}_9\mathrm{N}\mathrm{O}_2$ (right) along a curve of the form \eqref{experiment_interpolant} for different basis sets. }
    \label{fig:error_decay}
\end{figure}

\begin{table}[htpb]
    \centering
    \caption{Number of occupied orbitals $\text{n}_{\text{occ}}$ and basis functions $\text{n}_{\text{b}}$ for Glycine and Proline.}
    \begin{tabular}{llccccc}
        \toprule
        \textbf{Molecule} & \textbf{Property} & \textbf{STO-3G} & \textbf{STO-6G} & \textbf{3-21G} & \textbf{6-31G} & \textbf{cc-pVDZ} \\
        \midrule
        \multirow{2}{*}{Glycine} 
        & $\text{n}_{\text{occ}}$ & 20 & 20 & 20 & 20 & 20 \\
        & $\text{n}_{\text{b}}$     & 30 & 30 & 55 & 55 & 95 \\
        \midrule
        \multirow{2}{*}{Proline} 
        & $\text{n}_{\text{occ}}$ & 31 & 31 & 31 & 31 & 31 \\
        & $\text{n}_{\text{b}}$     & 49 & 49 & 90 & 90 & 157 \\
        \bottomrule
        \multicolumn{7}{l}{\footnotesize }
    \end{tabular}
\end{table}
As a total error measure, we use the \textit{mean log error} (MLE), which reads
\begin{align}\label{error_metric}
    E_{\text{MLE}}(d)=\frac{1}{|D|}\sum_{\mu \in D}\log(E_\mu(d)),
\end{align}
where $D\subset I$ is the test set that contains 50 equidistant points of the unit interval and $E_\mu(d)$ the relative $\ell^2$-error of the interpolant amplitude vector $\tilde{t}^d(\mu)=(\tilde{t}^d_\nu(\mu))_{\nu\in \mathscr{I}_\ell}$ of $d$ nodes with respect to the true coupled cluster amplitudes $t(\mu)=(t_\nu(\mu))_{\nu\in \mathscr{I}_\ell}$.
In Figure \ref{fig:error_decay} the $E_{\text{MLE}}$ is displayed for the amino acids glycine $\mathrm{C}_2\mathrm{H}_5\mathrm{N}\mathrm{O}_2$ and proline $\mathrm{C}_5\mathrm{H}_9\mathrm{N}\mathrm{O}_2$ for different basis sets as a function of the number of Chebyshev nodes. For either experiment and any basis set an exponential decay with respect of the number nodes is visible. The MLE error represents the mean decay of the relative error in the amplitudes across all geometries, while the offset indicates the variance of these decay rates over the full set of geometries.  Instead of using the interpolant \eqref{tensor_interpolant} as the actual coupled cluster solution, it can also be employed as an initial guess for the quasi-Newton method that is used in practice to solve the CCSD amplitude equations. Typically, the default initial guess for this method is given by the MP2 amplitudes, where all single-excitation amplitudes are set to zero and the double-excitation amplitudes are taken from second order perturbation theory. 

\begin{figure}[H]
\centering{\includegraphics[width=6in]{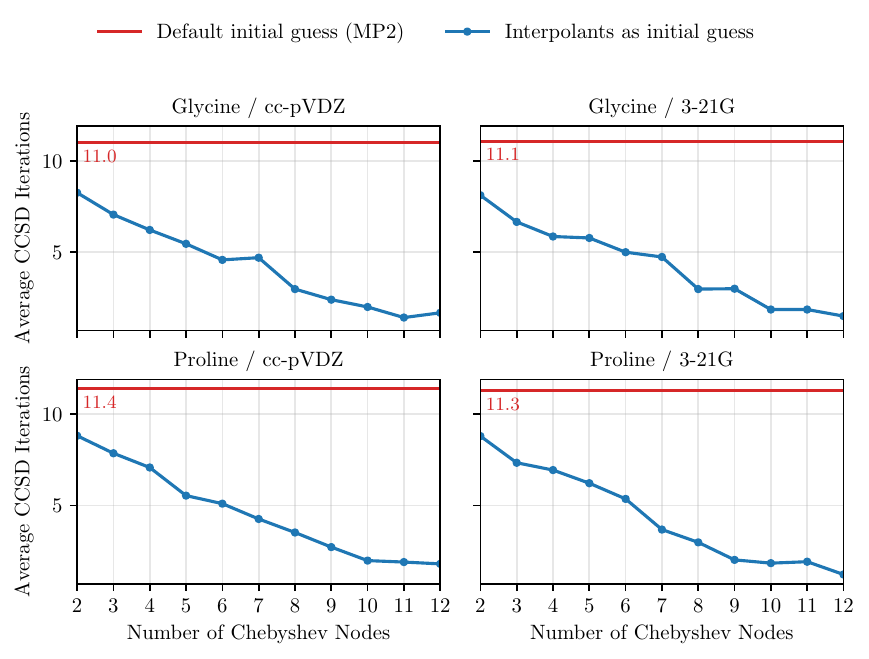}}
    \caption{The upper two plots show the average number of quasi-Newton steps required for convergence of the CCSD method across all parameters in the glycine experiments using the cc-pVDZ basis set (left) and the 3-21G basis set (right). The lower two plots show the corresponding results for the proline experiments.}
    \label{fig:iter_n_decay}
\end{figure}
Although the coupled-cluster interpolation ansatz approaches chemical accuracy exponentially quickly as the number of nodes increases, Figure \ref{fig:iter_n_decay} indicates that iteration counts in the quasi-Newton CCSD solver decrease significantly even at low node numbers. This suggests that high interpolation accuracy is not strictly required to obtain a high-quality initial guess. Finally, we investigated how well the interpolated coupled cluster amplitudes reproduce the correlation energy when it is computed directly from these approximate amplitudes.

\begin{figure}[H]
\centering{\includegraphics[width=6in]{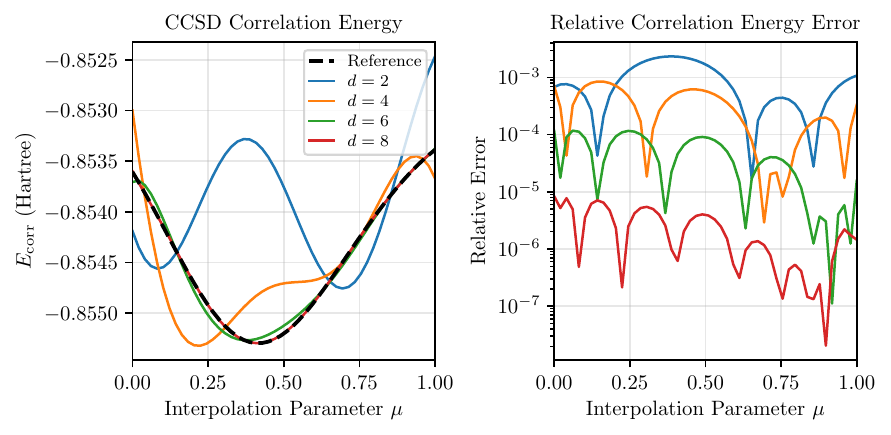}}
    \caption{Correlation energy of glycine with basis set cc-pVDZ along the trajectory $\Gamma$ for different numbers of Chebyshev nodes (left). The corresponding relative error between the reference correlation energy and the correlation energy reconstructed from interpolated amplitudes is shown on the right.}
    \label{fig:PES}
\end{figure}

In Figure \ref{fig:PES} the correlation energy for the experiment with glycine with respect to the cc-pVDZ basis set is shown. The plot on the right-hand side exhibits the characteristic spikes in the relative errors near the interpolation nodes. Notably, the expression of the CCSD correlation energy accumulates errors in a summation, which causes the spikes to dip lower as the number of nodes increases.

\section{Conclusions}\label{conclusion}
In this work, we examined the regularity of coupled cluster amplitudes in the canonical orbital basis with respect to variations in the nuclear coordinates. Because these amplitudes inherit their regularity properties from the Hartree–Fock level of theory, we also established corresponding results for the Hartree–Fock density and molecular orbitals, thereby justifying several key assumptions underlying the coupled cluster framework. The required regularity assumptions for the Hartree–Fock functional are typically satisfied in practice due to the widespread use of Gaussian-type orbitals in the chemistry community, which in fact often leads to analyticity.\\
\newline
A major drawback of canonical orbitals is the possibility of orbital energy crossings, which, in state-of-the-art quantum chemistry algorithms, can lead to index reorderings that compromise the regularity results in practice. To address this issue, a tensor transformation was introduced that not only restores regularity but also enables interpolation and extrapolation schemes for the coupled cluster amplitudes under variations in nuclear displacement, while remaining efficient in both time and memory scaling. Numerical experiments demonstrate an exponential decay in the interpolation error with an increasing number of Chebyshev nodes in the interpolation ansatz, in agreement with the theoretical regularity results derived in the preceding sections. Moreover, using the interpolants as initial guesses leads to a steady reduction in the number of required quasi-Newton steps in the CCSD method, indicating that even a small number of nodes is sufficient to significantly accelerate the CCSD iterative solver. In the subsequent works, we will apply the transformation ansatz in practice by developing an extrapolation ansatz that may be useful in both geometry optimization and molecular dynamics.

\subsection*{Acknowledgments} The authors sincerely thank Filippo Lipparini, Michele Nottoli and Paul Schwahn for their insightful discussions.
We acknowledge the support by the
Stuttgart Center for Simulation Science (SimTech).

\bibliographystyle{macros/plain-doi}
\bibliography{macros/journalabbr,macros/Literature}    

\end{document}